\documentclass[a4paper, 10pt, oneside, onecolumn]{article}

\RequirePackage{geometry}
\usepackage[utf8]{inputenc}
\usepackage[T1]{fontenc}
\usepackage{lmodern}
\usepackage{amsmath}
\usepackage{amssymb}
\usepackage{amsfonts}
\usepackage{amsthm}
\usepackage{xcolor}
\usepackage{hyperref}
\usepackage{cleveref}
\usepackage{mathtools}
\usepackage[autostyle=true]{csquotes}

\hypersetup{
	colorlinks=false,
	pdfborder={0 0 0},
	pdftitle={On the existence of global smooth solutions to the parabolic--elliptic Keller--Segel system with irregular initial data},
	pdfauthor={Frederic Heihoff},
	pdfkeywords={},
	bookmarksopen=true,
}

\renewcommand{\phi}{\varphi}

\newtheorem{base}{Base}[section]
\numberwithin{equation}{section}

\theoremstyle{plain}
\newtheorem{theorem}[base]{Theorem}
\newtheorem{lemma}[base]{Lemma}

\newtheorem{corollary}[base]{Corollary}

\theoremstyle{definition}

\newtheorem{remark}[base]{Remark}

\newcommand{\R}{\mathbb{R}}
\newcommand{\N}{\mathbb{N}}
\renewcommand{\d}{\,\mathrm{d}}
\newcommand{\laplace}{\Delta}
\newcommand{\grad}{\nabla}
\renewcommand{\div}{\nabla \cdot}
\renewcommand{\L}[1]{{L^{#1}(\Omega)}}
\newcommand{\defs}{\coloneqq}
\newcommand{\sfed}{\eqqcolon}
\newcommand{\stext}[1]{\;\;\text{ #1 }\;\;}
\newcommand{\eps}{\varepsilon}
\newcommand{\loc}{\mathrm{loc}}
\newcommand{\tmax}{{T_{\mathrm{max}}}}

\newcommand{\Mp}{\mathcal{M}_+(\overline{\Omega})}
\newcommand{\idata}{\mu}

\newif\ifclarification
\clarificationtrue

\newcommand\numberthis{\addtocounter{equation}{1}\tag{\theequation}}

\makeatletter
\g@addto@macro\bfseries{\boldmath}
\makeatother

\title{On the existence of global smooth solutions to the parabolic--elliptic Keller--Segel system with irregular initial data}
\author{
	Frederic Heihoff\footnote{fheihoff@math.uni-paderborn.de}\\
	{\small Institut f\"ur Mathematik, Universit\"at Paderborn,}\\
	{\small 33098 Paderborn, Germany}
}
\date{}

\begin{document}
	
\maketitle
\begin{abstract}
	\noindent
	We consider the parabolic--elliptic Keller--Segel system 
	\[
		\left\{ 
		\begin{aligned}
		u_t &= \Delta u - \chi \nabla \cdot (u \nabla v), \\
		0 &= \Delta v - v + u
		\end{aligned}
		\right. \tag{$\star$}
	\]
	in a smooth bounded domain $\Omega \subseteq \mathbb{R}^n$, $n\in\mathbb{N}$, with Neumann boundary conditions. We look at both chemotactic attraction ($\chi > 0$) and repulsion ($\chi < 0$) scenarios in two and three dimensions.
	\\[0.5em]
	The key feature of interest for the purposes of this paper is under which conditions said system still admits global classical solutions due to the smoothing properties of the Laplacian even if the initial data is very irregular. Regarding this, we show for initial data $\mu \in \mathcal{M}_+(\overline{\Omega})$ that, if either
	\begin{itemize}
		\item $n = 2$, $\chi < 0$ or
		\item $n = 2$, $\chi > 0$ and the initial mass is small or
		\item $n = 3$, $\chi < 0$ and $\mu = f \in L^p(\Omega)$, $p > 1$  
	\end{itemize}
	holds, it is still possible to construct global classical solutions to ($\star$), which are continuous in $t = 0$ in the vague topology on $M_+(\overline{\Omega})$.
	\\[0.5em]
\end{abstract}

\section{Introduction}

In this paper, we are primarily concerned with systems of partial differential equations used in the study of biological systems. More specifically, we are interested in systems modeling chemotaxis, the directed movement of cells along a chemical gradient. This use of partial differential equations in the biological study of chemotactic processes was mostly initiated by the seminal work of Keller and Segel in 1970 (cf.\ \cite{keller1970initiation}), in which Keller and Segel used them to model certain slime molds in an effort to understand their aggregation behavior. The popularity of this approach was further bolstered by the subsequent successful mathematical analysis of said model, which confirmed the presence of aggregation in the sense that under appropriate initial conditions the solutions to the system blow up in finite time while retaining their initial mass (cf.\ \cite{NagaiBlowupNonradialSolutions2001}, \cite{WinklerBlowUp}). This success in modeling and mathematical analysis has led to many more biological processes (and sometimes even comparable processes from other fields, such as criminology, cf.\ \cite{ShortCrime}) to be modeled in a similar fashion. For a survey, we refer the reader to \cite{Survey}. 
\\[0.5em]
As already alluded to, there exist many scenarios (generally dependent on the initial data and dimension of the domain), in which solutions to these kinds of chemotaxis models blow up in finite time (cf.\
\cite{HorstmannBoundednessVsBlowup2005}, \cite{JagerExplosionsSolutionsSystem1992}, 
 \cite{NagaiBlowupRadiallySymmetric1995}, \cite{NagaiBlowupNonradialSolutions2001}, \cite{SenbaParabolicSystemChemotaxis2001},
 \cite{WinklerBlowUp}). It is further known that in some of these scenarios this blowup takes the form of the solution converging to approximately a Dirac mass or some function of potentially very little regularity (cf.\ \cite{MR1970697}, \cite{MR3936129}). What we are now interested in is in a sense the opposite scenario, which has to our knowledge been less often discussed thus far. Namely, we consider the case of starting with initial data of very little regularity, such as a Dirac measure, and are then concerned with deriving whether or under which conditions there still exist sensible global classical solutions for such a model, which are still connected to the initial data in a reasonable fashion. This can be interpreted as essentially starting our analysis at the point in time when aggregation occurred and then investigating under which circumstances the model still yields sensible results from that point onward.
\\[0.5em]
The model we want to analyze in this regard will be a variation on the original (here somewhat simplified) Keller--Segel model (cf.\ \cite{keller1970initiation}):
\begin{align*}
\left\{
\begin{aligned}
u_t &= \laplace u - \div (u\grad v), \\
v_t &= \laplace v - v + u
\end{aligned}
\right.
\end{align*}
In this model, the function $u$ represents the density of the organism under consideration while the function $v$ represents the density of the attractant substance. Both are under the influence of diffusion modeled by the terms $\laplace u$ and $\laplace v$. The central term representing the chemotatic interaction is $\div (u \grad v)$. The remaining linear terms in the second equation then model the degradation of the attractant over time and the production of the attractant by the organisms, respectively. 
\\[0.5em]
We then consider the similar system (with the same roles for $u$ and $v$)
\begin{equation} \label{problem}
\left\{ 
\begin{aligned}
u_t &= \laplace u - \chi \div (u \grad v), \\
0 &= \laplace v - v + u,
\end{aligned}
\right.
\end{equation}
in which the second equation is only of elliptic type. This modification of the original Keller--Segel system is in fact not uncommon (cf.\ \cite{JagerExplosionsSolutionsSystem1992}, \cite{NagaiBlowupRadiallySymmetric1995}) and can be understood as reflecting that the chemical responds immediately everywhere to changes in the population density of the modeled organism as opposed to the organism density only influencing its time evolution and therefore having a delayed effect on it. Additionally, we also add a potentially negative coefficient $\chi$ to the chemotaxis term $\div (u \grad v)$, which models that the chemical substance cannot only be an attractant but also possibly a repellent (cf.\ \cite{LucaChemotacticSignalingMicroglia2003}, \cite{MarkPatterningNeuronalConnections1997}, \cite{PiniChemorepulsionAxonsDeveloping1993} for some biological processes involving chemorepulsion).

\paragraph{Main result.}
We consider the system (\ref{problem}) with Neumann boundary conditions 
\begin{equation} 
\left\{	 \label{boundary_conditions}
\begin{aligned}
\grad u(x,t) \cdot \nu &= 0 &&\;\;\;\; \text{ for all } x\in\partial\Omega, t > 0, \\
\grad v(x,t) \cdot \nu &= 0 &&\;\;\;\; \text{ for all } x\in\partial\Omega, t > 0
\end{aligned}
\right.
\end{equation}
in a bounded domain $\Omega \subseteq \R^n$, $n \in \{2,3\}$, with smooth boundary. Concerning the initial data, we assume $\idata$ to be an element of $\Mp$, the set of all positive Radon measures with the vague topology. The vague topology on $\Mp$ is characterized as follows (cf.\ \cite[Definition 30.1]{BauerMeasureAndIntegration}): A sequence $(\mu_n)_{n\in\N} \subseteq \Mp$ converges to $\mu \in \Mp$ in the vague topology if and only if
\[
	\int_{\overline{\Omega}} f \d\mu_n \rightarrow \int_{\overline{\Omega}} f \d\mu \;\;\;\; \text{ as }  n \rightarrow \infty \;\;\;\; \text{ for all } f \in C(\overline{\Omega}).
\]
For the purposes of this paper, whenever necessary we identify nonnegative functions $\phi \in L^1(\Omega)$ with the measure $\phi(x) \d x$, where $\d x$ represents the standard Lebesgue measure.
\\[0.5em]
Under these conditions, we then investigate some scenarios under which the smoothing properties of the Laplacian are sufficient to counteract the irregularity of the initial data and the destabilizing effects of the taxis term and therefore make it possible to still construct smooth solutions that attain the initial data in a sensible way. These scenarios are repulsive chemotaxis in two and three dimensions (with slightly higher regularity needed for the initial data in the three dimensional case) and attractive chemotaxis in two dimensions (with an additional initial mass condition). More precisely, we prove the following 
\begin{theorem} \label{theorem:main}
	Let $n\in\N$, $\chi \in \R$ and let $\Omega \subseteq \R^n$ be a bounded domain with smooth boundary. Then there exists a constant $C_m > 0$ such that the following holds:
	\\[0.5em] 
	Let $\idata \in \Mp$ be some initial data. If
	\begin{align}
		 &n = 2 \stext{ and }  \chi < 0 \label{scenario:2d+repulsion}\tag{S1}\\ 
		\text{ or }\;\;\;\; &n = 2 \stext{ and } \chi > 0 \stext{ and } \mu(\Omega) \leq C_m \label{scenario:2d+attraction}\tag{S2} \\
		\text{ or }\;\;\;\; &n = 3 \stext{ and } \chi < 0 \stext{ and } \mu = f \text{ for some } f \in L^p(\Omega) \text{ with } p > 1 \label{scenario:3d+repulsion}\tag{S3},
	\end{align} 
	then there exist functions
	\[
	\left\{
	\begin{aligned}
	u &\in C^{2,1}(\overline{\Omega}\times(0,\infty)), \\
	v &\in C^{2,0}(\overline{\Omega}\times(0,\infty))
	\end{aligned}
	\right.
	\]
	that solve (\ref{problem}) on $\Omega\times(0,\infty)$ with boundary conditions (\ref{boundary_conditions}) classically and attain the initial data $\mu$ in the following way: 
	\begin{equation*}
		u(\cdot, t) \rightarrow \idata \;\;\;\; \text{ in } \Mp \text{ as } t\searrow 0.
	\end{equation*}
\end{theorem}
\paragraph{Prior work.}
To give some context for our existence result, we will now give a brief overview over some notable prior work in this area. 
\\[0.5em]
We begin by reviewing some results concerning smooth initial data as opposed to the irregular initial data considered here. This of course makes the construction of solutions easier and thus, especially in the repulsive case, there are quite strong existence results available. Namely, it can be shown that problems of type (\ref{problem}) with $\chi < 0$ and smooth initial data have global classical solutions in domains of arbitrary dimension, which converge to their steady states at an exponential rate (cf.\ \cite{SemiConductorExist}, \cite{SemiConductorAsymptotic}, \cite{CompetingChemotaxis}). The attractive case is somewhat more complex regarding existence theory as existence here generally depends on properties of the initial data. In two dimensions existence centrally depends on the initial mass, while in higher dimensions existence can only be ensured for much stronger initial data smallness conditions (cf.\ \cite{MR1046835}, \cite{MR1361006}, \cite{MR1887324}, \cite{MR1623326}). This already suggests that the mass condition for the two-dimensional attractive case (S2) is certainly necessary. 
\\[0.5em]
While we are not as concerned with the parabolic-parabolic case, we still want to mention that similar, but maybe not always quite as strong, results are available in this case as well (cf.\ \cite{MR2549326} for existence results in the repulsive case and \cite{HorstmannBoundednessVsBlowup2005}, \cite{MR1610709}, \cite{MR1893940}, \cite{WinklerBlowUp} for discussions of the attractive case to only list a few). We again refer to the survey \cite{Survey} for a broader overview.
\\[0.5em] 
There have also been efforts to analyze chemotaxis systems with two associated elliptic or parabolic equations with one modeling an attractant and the other a repellent with the key result being that existence of solutions is ensured as long are the repellent forces as stronger than their attractive counterparts (cf.\ \cite{CompetingChemotaxis}). We mention this result as it already illustrates how repulsive chemotaxis generally poses less of problem when constructing solutions as opposed to its attractive counterpart, which is in a sense mirrored in our result.
\\[0.5em]
We now transition to some prior work concerning results about systems similar to (\ref{problem}) with irregular initial data. For the two-dimensional whole space case, there are in fact existence results available for a system similar to (\ref{problem}) with measure valued initial data with results e.g.\  based on methods from harmonic analysis (cf.\ \cite{MR3411404}, \cite{MR2483520}). Moreover, a weak solution construction on the torus is presented in \cite{MR1909263}. In \cite{MR4022112}, a system similar to (\ref{problem}) with an added logistic source is discussed under the assumption that the initial data is radially symmetric and has a singularity in $x=0$ but is otherwise fairly regular. As all of these results restrict themselves to very specific settings in an effort to make use of these restrictions to construct solutions, we can generally not translate the methods employed in them to our setting, which is concerned with general bounded domains and fairly general initial data.
\\[0.5em] 
As for the parabolic-parabolic case for the classic Keller--Segel model in bounded domains, it is know that smooth solutions with irregular initial data exist in either the one dimensional case (cf.\ \cite{MR3905266}) or in the two dimensional case with an added logistic source term acting as an additional regularizing factor (cf.\ \cite{LankeitIrregularInitialData}).  
\paragraph{Approach.}
As is often the case when constructing solutions, our basic approach will be looking at approximate solutions $(u_\eps, v_\eps)_{\eps \in (0,1)}$ that solve a certain regularized version of (\ref{problem}), which can easily be seen to admit global classical solutions, and then gaining our desired solutions $(u,v)$ as limits of $(u_\eps, v_\eps)$ as $\eps \searrow 0$. The key regularizations employed by us for this approach are approximating the initial data by smooth functions and replacing the linear growth term $u$ in the second equation by a term that is bounded independent of $u$, but approaches the original linear term as $\eps \searrow 0$. For the exact system, see (\ref{approx_system}).
\\[0.5em]
Our next step then is deriving a priori estimates for the approximate solutions that do not depend on $\eps$. This is made particularly challenging due to the fact that we cannot rely on much initial data regularity, which is normally a key part of most testing or semigroup based approaches. When using testing based methods, this is due to the fact that the a priori information gained using them is often based on deriving ordinary differential equations with at best a linear decay term for some of the norms of the solution components and then using comparison arguments, which still take the norm of the initial data into account. 
\\[0.5em]
As such, our testing approaches are focused on deriving ordinary differential equations for terms of the form $\int_\Omega u_\eps^p$, which have a superlinear decay term, see \Cref{lemma:the_central_ineq} and \Cref{lemma:the_central_ineq_attraction}. It is the arguments used in the proofs of both of these lemmas where most of the restrictions on the allowed dimension $n$ and values of $\chi$, as well as the initial mass restriction in \Cref{theorem:main} originate from (Only the need for higher regularity of the initial data in the three dimensional case stems from a later argument in \Cref{lemma:ugradv_integrability}, which is apparently necessary to ensure that the constructed solutions still attain the initial data in a sensible fashion). As proven in \Cref{lemma:ode_bound}, these ordinary differential equations then allow us to gain uniform (in regards to $\eps$) $\|u_\eps\|_\L{p}$ bounds for all $p \in [1,\infty)$ on $(t_0, \infty)$, $t_0 > 0$, for the approximate solutions, which by standard bootstrap arguments combined with some compact embedding properties of Hölder spaces and standard regularity theory yield sufficiently regular classical solutions $(u,v)$ of (\ref{problem}) and (\ref{boundary_conditions}) along a suitable sequence $\eps_j \searrow 0$. Additionally, the same argument also gives us certain uniform time integrability properties for $\int_\Omega u_\eps^p$ on $(0,1)$, which are then used in \Cref{section:inital_data_regularity} to conclude that $\int_\Omega u_{\eps t} \phi$ has similar uniform time integrability properties. This implies that the approximate solutions are continuous in $t = 0$ regarding the vague topology in a uniform sense. In \Cref{lemma:initial_data_continuity}, we then use this to argue that this continuity therefore survives the limit process and is thus still present in the actual solutions $(u,v)$.

\section{A regularized version of (\ref{problem}) with approximated initial data}

The key to our construction of solutions to the system (\ref{problem}) with irregular initial data will lie in framing said solutions as the limits of approximate solutions to a similar system, which is regularized in two key ways to make the existence of classical solutions much more obvious. The first regularization we employ will be to approximate the measure-valued initial data by smooth functions while the second is replacing the linear growth term $u$ in the second equation of (\ref{problem}) by a term that is always uniformly bounded independent of the value of $u$ but results in the original term after a limit processes.
\\[0.5em]
More precisely, we will use the following approximate system with smooth initial data $u_{0, \eps}$:
\begin{equation}\label{approx_system}
\left\{
	\begin{aligned}
		{u_\eps}_t &= \laplace u_\eps - \chi \div (u_\eps \grad v_\eps) \;\;\;\; &&\text{ on } \Omega\times(0,\infty), \\
		0 &= \laplace v_\eps - v_\eps + \tfrac{u_\eps}{1+\eps u_\eps}\;\;\;\; &&\text{ on } \Omega\times(0,\infty), \\
		\grad u_\eps \cdot \nu &= 0, \;\; \grad v_\eps \cdot \nu = 0\;\;\;\; &&\text{ on } \partial\Omega\times(0,\infty), \\
		u(\cdot, 0) &= u_{0,\eps} \;\;\;\; &&\text{ on } \Omega
	\end{aligned}	
\right.
\end{equation}
For this system, it is fairly straightforward to construct unique global classical solutions by first constructing local solutions by standard contraction mapping methods (cf.\ \cite{LocalExistenceInSimilarSetting}) and then arguing that finite-time blowup is in fact impossible. The latter step is made easier by the fact that $\frac{u_\eps}{1+\eps u_\eps} \leq \frac{1}{\eps}$, which immediately gives us quite strong bounds for the second solution component. 
\\[0.5em]
As both this approach and the employed regularization are  quite standard, we will only give the following existence argument in brief.
\begin{lemma} \label{lemma:approx_existence}
	Let $n \in \N$, $\chi \in \R$ and let $\Omega \subseteq \R^n$ be a bounded domain with smooth boundary. Then for $\eps \in (0,1)$ and nonnegative $u_{0,\eps} \in C^\infty(\overline{\Omega})$, there exist nonnegative functions
	\begin{align*}
		u_\eps &\in C^{2,1}(\overline{\Omega}\times(0,\infty))\cap C^0(\overline{\Omega}\times[0,\infty)), \\
		v_\eps &\in C^{2,1}(\overline{\Omega}\times(0,\infty)),
	\end{align*}
	that are a global classical solution of (\ref{approx_system}) with the following additional mass conservation property:
	\begin{equation} \label{eq:mass_conservation}
		\int_\Omega u_\eps(\cdot, t) = \int_\Omega u_{0,\eps} \;\;\;\; \text{ for all } t > 0.
	\end{equation}
\end{lemma}
\begin{proof}
	By an adaption of standard contraction mapping and maximum principle arguments used in similar settings as seen e.g.\ in \cite[Proposition 3.1]{LocalExistenceInSimilarSetting}, we gain a maximal $\tmax \in (0,\infty]$ and nonnegative functions
	\begin{align*}
		u_\eps &\in C^{2,1}(\overline{\Omega}\times(0,\tmax))\cap C^0(\overline{\Omega}\times[0,\tmax)), \\
		v_\eps &\in C^{2,1}(\overline{\Omega}\times(0,\tmax))
	\end{align*}
	that are a classical solution to (\ref{approx_system}) on $[0,\tmax)$ and adhere to (\ref{eq:mass_conservation}) on $[0,\tmax)$. As a consequence of this standard construction, we further know that, if $\tmax<\infty$, then $\limsup_{t\nearrow\tmax} \|u_\eps(\cdot, t)\|_\L{\infty} = \infty$. 
	\\[0.5em]
	We will now briefly sketch why this blowup of the $L^\infty(\Omega)$ norm of $u_\eps$ in finite time is in fact impossible, which is of course sufficient to complete this proof. First and foremost due to the fact that $\frac{u_\eps}{1+\eps u_\eps} \leq \frac{1}{\eps}$, standard elliptic regularity theory (cf.\ \cite[Theorem 19.1]{FriedmanPDE}) applied to the second equation in (\ref{approx_system}) immediately yields a $W^{2,p}(\Omega)$ bound for $v_\eps$ on $[0,\tmax)$ for all $p \in (1,\infty)$, which by the well-known embedding properties of Sobolev spaces in turn translates to a $W^{1,\infty}(\Omega)$ bound for $v_\eps$ on $[0,\tmax)$. With this, we can now use the variation-of-constants representation of $u_\eps$ relative to the semigroup $(e^{t\laplace})_{t > 0}$ in combination with well-known smoothness estimates (cf.\ \cite[Lemma 1.3]{WinklerSemigroupRegularity}) of said semigroup, the maximum principle and the Hölder inequality to gain constants $C, \lambda > 0$, $\alpha \in (0,1)$ and $q \in (n,\infty)$ such that 
	\begin{align*}
	&\|u_\eps(\cdot, t)\|_\L{\infty}  \\
	=& \left\| e^{t\laplace}u_{0,\eps} - \chi \int_0^t e^{(t-s)\laplace} \div \left( u_\eps \grad v_\eps \right) \d s \right\|_\L{\infty} \\
	 \leq& \|u_{0,\eps}\|_\L{\infty} + C|\chi|\int_0^t (1+(t-s)^{-\frac{1}{2} - \frac{n}{2q}}) e^{-\lambda(t-s)} \|u_\eps\|_\L{q} \|\grad v_\eps\|_\L{\infty} \d s \\
	\leq& \|u_{0,\eps}\|_\L{\infty} + C|\chi|\|u_\eps\|_{L^\infty(\Omega\times[0,T])}^\alpha \int_0^t (1+(t-s)^{-\frac{1}{2} - \frac{n}{2q}}) e^{-\lambda(t-s)} \|u_\eps\|^{1-\alpha}_\L{1} \|\grad v_\eps\|_\L{\infty} \d s
	\end{align*}
	for all $T\in[0,\tmax)$ and $t \in [0,T]$. As the remaining integral term above is bounded independent of $t$ by prior arguments, taking the supremum over $t\in[0,T]$ then immediately yields an $L^\infty(\Omega)$ bound for $u_\eps$ on $[0,T]$ for all $T\in [0,\tmax)$, which is in fact independent of $T$. This implies that $\|u_\eps\|_{L^\infty(\Omega\times(0,\tmax))} < \infty$ and therefore completes the proof.
\end{proof}
\noindent Given that we have now established the necessary existence theory for our approximate solutions, let us fix some functions, measures and parameters for the remainder of this paper in a effort to not unnecessarily clutter later results and arguments. 
\\[0.5em]
First, let $n\in\N$, $\chi \in \R$ be fixed and let $\Omega \subseteq \R^n$ always be a bounded domain with smooth boundary. We then fix some initial data $\mu \in \Mp$ and a family $(u_{0,\eps})_{\eps \in (0,1)} \subseteq C^\infty(\overline{\Omega})$ of nonnegative functions with
\begin{equation} \label{approx_mass_conservation}
	\int_\Omega u_{0,\eps} = \idata(\overline{\Omega}) \sfed m
\end{equation}
that approximate $\mu$ in the following way:
\begin{equation} \label{approx_initial_data}
	u_{0,\eps} \rightarrow \idata \;\;\;\; \text{ in } \Mp \text{ as } \eps \searrow 0
\end{equation}
\begin{remark}
	Let us give a brief argument as to how such an approximation of Radon measures can be achieved: 
	It is fairly easy to see that approximating Dirac measures $\delta_x$ by smooth functions in this way is indeed possible given sufficient boundary regularity (e.g.\ by using $f_\eps(y) \defs C(\eps)e^{-\frac{1}{\eps}|x-y|^2}$ with $C(\eps) > 0$ some normalization constant). This implies that the Dirac measures are contained in the closure (relative to the vague topology) of the set $\mathcal{F} \defs \{ \phi \in C^\infty(\overline{\Omega}) \;|\; \int_\Omega \phi = 1, \phi \geq 0 \}$. Further, one can show that the Dirac measures are the extreme points of the convex set of probability measures $\mathcal{M}_1(\overline{\Omega}) \subseteq \Mp$, which is compact in the vague topology. This makes it accessible to the Krein--Millman theorem (cf.\ \cite[Theorem 3.23]{RudinFunctionalAnalysis}) implying that 
	\[
	\mathcal{M}_1(\overline{\Omega}) = \overline{\text{conv}(\{ \delta_x \,|\, x \in \Omega \})} \subseteq \overline{\mathcal{F}} \subseteq \mathcal{M}_1(\overline{\Omega})
	\] 
	and therefore that $\mathcal{M}_1(\overline{\Omega}) = \overline{\mathcal{F}}$ (see also \cite[Corollary 30.5]{BauerMeasureAndIntegration}). As $\Mp$ is metrizable (cf.\ \cite[Theorem 31.5]{BauerMeasureAndIntegration}), this is sufficient to gain our desired approximation after a straightforward scaling argument.
\end{remark}
\noindent 
If $\mu = f$ for some $f \in L^p(\Omega)$, $p > 1$, we further assume that
\begin{equation}\label{eq:better_approx}
	u_{0,\eps} \rightarrow f \;\;\;\; \text{ in } L^p(\Omega)\text{ as } \eps \searrow 0.
\end{equation} 
This additional approximation property can be achieved by standard methods for approximating $L^p(\Omega)$ functions by smooth functions combined with a straightforward normalization argument to ensure (\ref{approx_mass_conservation}).
\\[0.5em]
\noindent
According to \Cref{lemma:approx_existence}, we then fix a nonnegative global classical solution $(u_\eps, v_\eps)$ to (\ref{approx_system}) with initial data $u_{0,\eps}$ for each $\eps \in (0,1)$.

\section{An initial data independent estimate for an ordinary differential equation}

Deriving ordinary differential equations for key norms by testing partial differential equations with carefully chosen functions is often one of the first steps in the process of gaining sufficient a priori information about said partial differential equations. But, if the decay terms in these ordinary differential equations are not sufficiently strong, the estimates gained from them are often still dependent on the initial data. In general, this poses not much of a problem as the initial data is in many cases assumed to be fairly regular, but in our case this means estimates of this type are mostly useless because our set of approximate initial data $(u_{\eps,0})_{\eps \in (0,1)}$ is in general not even bounded in any $L^p(\Omega)$ with $p > 1$. As such, the ordinary differential equations we will derive in this paper for norms of our approximate solutions will need to have decay terms strong enough to allow for initial data (and therefore $\eps$) independent estimates. Note that, by their very nature, these estimates will always break down for $t \searrow 0$, but may still yield certain integrability properties up to zero instead.
\\[0.5em]
For this purpose, we will in this section consider the initial value problem 
\begin{equation}
\left\{
\begin{aligned}
y'(t) &= -A y^\alpha(t) + B, \;\;\;\; t > 0, \\
y(0) &= y_0
\end{aligned}
\right. \label{eq:helpful_ode}
\end{equation}
with $y_0 \geq 0, A > 0, B \geq 0, \alpha > 1$ and prove that the above superlinear decay term is in fact enough to grant us initial data independent estimates for its solution. We will even quantify somewhat how severely these estimates deteriorate for $t \searrow 0$.
\\[0.5em]
While the proof presented here is fairly straightforward, we will still present the argument leading to the following estimate for (\ref{eq:helpful_ode}) in full due to its important role for the central results of this paper.
\begin{lemma} \label{lemma:ode_bound}
	For each $A > 0, B\geq 0$ and $\alpha > 1$, there exists $C \equiv C(A,B,\alpha) > 0$ such that the following holds: The solution $y \in C^0([0,\infty)) \,\cap\, C^1((0,\infty))$ of (\ref{eq:helpful_ode}) with initial data $y_0$ has the property
	\begin{equation*}
		y(t) \leq Ct^\frac{1}{1-\alpha} + C \;\;\;\; \text{ for all } t > 0.
	\end{equation*}
\end{lemma}

\begin{proof}
	We fix $A > 0, B \geq 0, \alpha > 1$ and initial data $y_0 \geq 0$.
	\\[0.5em]
	That the solution $y$ exists locally and is unique is ensured by the Picard--Lindelöf theorem. Global existence and nonnegativity then follow by comparing with a sufficiently large constant or zero.
	\\[0.5em]
	If $y_0 \leq \left( B/A \right)^\frac{1}{\alpha}$, then by comparison with a constant function of value $\left( B/A \right)^\frac{1}{\alpha}$, we immediately know that
	\begin{equation}\label{eq:small_initial_data}
		y(t) \leq \left( \tfrac{B}{A} \right)^\frac{1}{\alpha} \;\;\;\; \text{ for all } t > 0.
	\end{equation}
	As this is already sufficient for our desired result, we will now focus on the remaining case $y_0 > \left( B/A \right)^\frac{1}{\alpha}$. In this case, we immediately gain
	\begin{equation*}
		y(t) > \left( \tfrac{B}{A} \right)^\frac{1}{\alpha} \;\;\;\; \text{ for all } t > 0
	\end{equation*}
	by essentially the same comparison argument. Given this, we define $z(t) \defs y(t) - \left( B/A \right)^\frac{1}{\alpha} > 0$ for all $t \geq 0$ and then note that
	\begin{equation*}
		z'(t) = y'(t) = -Ay^\alpha(t) + B = -A \left( y(t) - \left( \tfrac{B}{A} \right)^\frac{1}{\alpha} + \left( \tfrac{B}{A} \right)^\frac{1}{\alpha} \right)^\alpha + B \leq -A z^\alpha(t)  \;\;\;\; \text{ for all } t > 0.
	\end{equation*}
	By now comparing $z$ with the explicit solution to the initial value problem $w' = -Aw^\alpha, w(0) = z(0)$, we can conclude that
	\begin{equation*}
		z(t) \leq \left(A(\alpha - 1)t + z^{1-\alpha}(0)\right)^{\frac{1}{1-\alpha}} \leq (A(\alpha-1))^\frac{1}{1-\alpha}\, t^\frac{1}{1-\alpha}  \;\;\;\; \text{ for all }  t > 0
	\end{equation*}
	and therefore that
	\begin{equation}\label{eq:big_initial_data}
	y(t) \leq (A(\alpha-1))^\frac{1}{1-\alpha}\, t^\frac{1}{1-\alpha} + \left(\tfrac{B}{A}\right)^{\frac{1}{\alpha}} \;\;\;\; \text{ for all }  t > 0.
	\end{equation}
	As we have now covered all necessary cases, combining (\ref{eq:small_initial_data}) and (\ref{eq:big_initial_data}) completes the proof with \[
	C \defs \max\left( \, (A(\alpha - 1))^\frac{1}{1-\alpha}, \, \left(B/A\right)^{\frac{1}{\alpha}} \, \right). \qedhere
	\]
\end{proof}

\section{Deriving our central differential inequality for $\int_\Omega u_\eps^p$}

This next section is now devoted to deriving exactly the type of differential inequalities discussed in the previous one for terms of the form $\int_\Omega u_\eps^p$, $p \in (1,\infty)$, with constants independent of $\eps$. The basic approach for this is testing the first equation in (\ref{approx_system}) with $u_\eps^{p-1}$, employing partial integration and applying the second equation in (\ref{approx_system}) to the resulting $\laplace v_\eps$ terms. Due to the change of sign of the terms originating from the second equation depending on the sign of $\chi$ (resulting in different problematic terms), we will need to treat the repulsive case ($\chi < 0$, see \Cref{lemma:the_central_ineq}) and attractive case ($\chi > 0$, see \Cref{lemma:the_central_ineq_attraction}) somewhat separately. Nonetheless, we still manage to achieve essentially the same result for both barring some additional restrictions on the initial data or dimension.
\\[0.5em]
Before we dive into the actual derivation of the central lemmas of this section, we will first establish some preliminary results about the approximate solutions. The first such result is the derivation of some bounds for the second solution component $v_\eps$. For this, we employ a classic result from elliptic regularity theory dealing with $L^1(\Omega)$ source terms (cf.\ \cite{BrezisStraussEllipticL1Regularity}) to the second equation in (\ref{approx_system}) to gain baseline bounds for later interpolation arguments. Further, we also test the same equation with $v_\eps^{r-1}$ to derive some additional estimates, which will prove helpful when applying some other regularity results later on.

\begin{lemma} \label{lemma:starting_point}
	For each $p \in [1, \frac{n}{n-1})$, there exists $C_1 \equiv C_1(p) > 0$  such that
	\[
		\|v_\eps(\cdot,t)\|_{W^{1,p}(\Omega)}\leq C_1
	\]
	for all $t > 0$ and $\eps \in (0,1)$ and therefore, for each $q \in [1,\frac{n}{n-2})$, there exists $C_2 \equiv C_2(q) > 0$ such that
	\[
		\|v_\eps(\cdot,t)\|_\L{q} \leq C_2
	\]
	for all $t > 0$ and $\eps \in (0,1)$. Further,
	\begin{equation}\label{eq:v_eps_leq_u_eps}
		\|v_\eps(\cdot,t)\|_\L{r} \leq \left\|\frac{u_\eps(\cdot,t)}{1+\eps u_\eps(\cdot, t)}\right\|_\L{r} 
	\end{equation}
	for all $r \in [1,\infty]$, $t > 0$ and $\eps \in (0,1)$.
\end{lemma}
\begin{proof}
	Given the mass conservation property (\ref{eq:mass_conservation}), we can apply elliptic regularity theory that can be used with $L^1(\Omega)$ source terms (cf.\ \cite[Lemma 23]{BrezisStraussEllipticL1Regularity}) to the operator $-\laplace + 1$ to gain the desired uniform $W^{1,p}(\Omega)$ bounds for the second solution component for all $p \in [1,\frac{n}{n-1})$. The remaining  $L^q(\Omega)$ bounds then follow from the Sobolev embedding theorem (cf.\ \cite[Theorem 2.72]{EllipticFunctionSpaces}).
	\\[0.5em]
	For $r \in [1,\infty)$, the inequality (\ref{eq:v_eps_leq_u_eps}) is a consequence of testing the second equation in (\ref{approx_system}) with $v_\eps^{r-1}$ and integrating by parts to first gain $\int_\Omega v_\eps^r \leq \int_\Omega \frac{u_\eps}{1+\eps u_\eps} v_\eps^{r-1}$, which implies $\int_\Omega v_\eps^r \leq \int_\Omega ( \frac{u_\eps}{1+\eps u_\eps} )^r$ due to Young's inequality and therefore (\ref{eq:v_eps_leq_u_eps}) for all $t > 0$ and $\eps \in (0,1)$. The case $r = \infty$ then follows by taking the limit $r \nearrow \infty$ in (\ref{eq:v_eps_leq_u_eps}).
\end{proof}\noindent
As our second preliminary result of this section, we now further derive two interpolation inequalities for the first solution component $u_\eps$ based on a slightly extended variant of the Gagliardo--Nirenberg inequality found in \cite{LankheitGNI}, which are used in both the repulsive and attractive case.
\begin{lemma}\label{lemma:u_interpol}
	For each $p \in (1,\infty)$, there exists $C \equiv C(p) > 0$ such that
	\begin{equation}
		\int_\Omega u_\eps^{p + \frac{2}{n}} \leq C m^\frac{2}{n} \int_\Omega |\grad u_\eps^\frac{p}{2}|^2 + C m^{p+\frac{2}{n}}\label{eq:u_interpol_1}
	\end{equation}
	and
	\begin{equation}
		\left( \int_\Omega u_\eps^{p} \right)^{1+\frac{2}{n(p-1)}} \leq C m^\frac{2p}{n(p-1)} \int_\Omega |\grad u_\eps^\frac{p}{2}|^2 + C m^{p + \frac{2p}{n(p-1)}} \label{eq:u_interpol_2}
	\end{equation}
	for all $t > 0$ and $\eps \in (0,1)$ with $m$ as defined in (\ref{approx_mass_conservation}).
\end{lemma}
\begin{proof}
	Fix $p \in (1,\infty)$. As then
	\[
	\frac{1}{2+\frac{4}{np}} + \frac{1}{n} \geq \frac{1}{2+\frac{4}{n}} + \frac{1}{n} = \frac{2 + \frac{4}{n} + n }{2n + 4} \geq\frac{2 + n }{2n + 4} = \frac{1}{2}\;\;\text{ and }\;\; \frac{2}{p} \leq 2 \leq 2 + \frac{4}{np},
	\]
	we can use the Gagliardo--Nirenberg inequality, or rather a variant of it from \cite[Lemma 2.3]{LankheitGNI}, which allows for some of the parameters to be from the interval $(0,1)$ in a way not covered by the original inequality, to gain $K_1 > 0$ such that
	\[
	\|u_\eps^\frac{p}{2}\|_\L{2 + \frac{4}{np}} \leq K_1\| \grad u_\eps^\frac{p}{2}\|^\alpha_\L{2} \|u_\eps^\frac{p}{2}\|^{1-\alpha}_\L{\frac{2}{p}} + K_1\|u_\eps^\frac{p}{2}\|_\L{\frac{2}{p}} 
	\]
	for all $t > 0$ and $\eps \in (0,1)$ with 
	\[
	\alpha = \frac{\frac{p}{2} - \frac{1}{2+\frac{4}{np}}}{\frac{p}{2} + \frac{1}{n} - \frac{1}{2}} = \frac{\;\frac{p + \frac{2}{n} - 1}{2+\frac{4}{np}}\;}{ \frac{1}{2}( p + \frac{2}{n} - 1)} = \frac{2}{2+\frac{4}{np}}.
	\]
	This then implies that 
	\begin{align*}
	\int_\Omega u_\eps^{p+\frac{2}{n}} = \|u_\eps^\frac{p}{2}\|^{2+\frac{4}{np}}_\L{2 + \frac{4}{np}} &\leq K_2 \|\grad u^\frac{p}{2}_\eps\|^2_\L{2} \|u^\frac{p}{2}_\eps\|^{\frac{4}{np}}_\L{\frac{2}{p}} + K_2\|u_\eps^\frac{p}{2}\|^{2+\frac{4}{np}}_\L{\frac{2}{p}} \\
	&= K_2 m^\frac{2}{n} \int_\Omega |\grad u_\eps^\frac{p}{2}|^2 + K_2m^{p+\frac{2}{n}} \numberthis \label{eq:u_interpol_1_proto}
	\end{align*}
	with $K_2 \defs (2K_1)^{2+\frac{4}{np}}$ due to the mass conservation property seen in \Cref{lemma:approx_existence}.
	\\[0.5em]
	Because moreover
	\[
	\frac{1}{n} + \frac{1}{2} \geq \frac{1}{2} \;\;\text{ and }\;\; \frac{2}{p} \leq 2,
	\]
	the same Gagliardo--Nirenberg type inequality from reference \cite{LankheitGNI} is again applicable and gives us $K_3 > 0$ such that
	\[
	\|u_\eps^{\frac{p}{2}}\|_\L{2} \leq K_3 \|\grad u_\eps^\frac{p}{2}\|_\L{2}^\beta \|u_\eps^\frac{p}{2}\|^{1-\beta}_\L{\frac{2}{p}} + K_3 \|u_\eps^\frac{p}{2}\|_\L{\frac{2}{p}}
	\]
	for all $t > 0$ and $\eps \in (0,1)$ with 
	\[
	\beta = \frac{\frac{p}{2} - \frac{1}{2}}{\frac{p}{2} + \frac{1}{n} - \frac{1}{2}} = \frac{1}{1 + \frac{2}{n(p-1)}}.
	\]
	This implies that
	\begin{align*}
	\left( \int_\Omega u_\eps^{p} \right)^{1+\frac{2}{n(p-1)}} = \|u_\eps^\frac{p}{2}\|^{2\left(1+\frac{2}{n(p-1)}\right)}_\L{2}
	&\leq  K_4\| \grad u_\eps^\frac{p}{2}\|^2_\L{2} \|u_\eps^\frac{p}{2}\|^\frac{4}{n(p-1)}_\L{\frac{2}{p}} + K_4\|u_\eps^\frac{p}{2}\|^{2\left(1+\frac{2}{n(p-1)}\right)}_\L{\frac{2}{p}}
	\\
	&= K_4 m^\frac{2p}{n(p-1)} \int_\Omega |\grad u_\eps^\frac{p}{2}|^2 + K_4 m^{p + \frac{2p}{n(p-1)}} \numberthis \label{eq:u_interpol_2_proto}
	\end{align*}
	for all $t > 0$ and $\eps \in (0,1)$ with $K_4 \defs (2K_3)^{2\left(1+\frac{2}{n(p-1)}\right)}$ due to the mass conservation property seen in \Cref{lemma:approx_existence}.
	\\[0.5em] 
	Combining (\ref{eq:u_interpol_1_proto}) and (\ref{eq:u_interpol_2_proto}) then completes the proof.
\end{proof}
\noindent Having established all the necessary preliminaries, we will now begin deriving the core results of this section by first considering the repulsive case ($\chi < 0$).

\begin{lemma} \label{lemma:the_central_ineq}
Let $p \in (1,\infty)$. If $\chi < 0$ and $n \in \{2,3\}$, then there exists $C \equiv C(p) > 0$ such that
\begin{equation}\label{eq:the_central_ineq}
	\frac{\d }{\d t}\int_\Omega u_\eps^p \leq - C\left( \int_\Omega u_\eps^p \right)^{1+\frac{2}{n(p-1)}} + C
\end{equation}
for all $t > 0$ and $\eps \in (0,1)$.
\end{lemma}
\begin{proof}
	Let $p \in (1,\infty)$. Then testing the first equation in (\ref{approx_system}) with $u_\eps^{p-1}$ and integrating by parts yields
	\begin{align*}
	\frac{1}{p(p-1)} \frac{\d}{\d t} \int_\Omega u_\eps^p =& -\int_\Omega |\grad u_\eps|^2 u_\eps^{p-2} + \chi\int_\Omega (\grad u_\eps \cdot \grad v_\eps) \,u_\eps^{p-1}  \\
	=& -\frac{4}{p^2}\int_\Omega |\grad u_\eps^\frac{p}{2}|^2 - \frac{|\chi|}{p}\int_\Omega \grad u^{p}_\eps \cdot \grad v_\eps \\
	=&-\frac{4}{p^2}\int_\Omega |\grad u_\eps^\frac{p}{2}|^2 + \frac{|\chi|}{p}\int_\Omega u^p_\eps \laplace v_\eps \\
	=&-\frac{4}{p^2}\int_\Omega |\grad u_\eps^\frac{p}{2}|^2 + \frac{|\chi|}{p}\int_\Omega u^p_\eps v_\eps - \frac{|\chi|}{p} \int_\Omega \frac{u^{p+1}_\eps}{1+\eps u_\eps} \\
	\leq&-\frac{4}{p^2}\int_\Omega |\grad u_\eps^\frac{p}{2}|^2 + \frac{|\chi|}{p}\int_\Omega u^p_\eps v_\eps \numberthis \label{eq:u_test}
	\end{align*}
	for all $t > 0$ and $\eps \in (0,1)$. Using the interpolation property (\ref{eq:u_interpol_1}) from \Cref{lemma:u_interpol}, we can now fix $K_1 > 0$ such that
    \begin{align*}
		\int_\Omega u_\eps^{p+\frac{2}{n}} \leq K_1 m^\frac{2}{n} \int_\Omega |\grad u_\eps^\frac{p}{2}|^2 + K_1 m^{p+\frac{2}{n}} \leq K_2 \int_\Omega |\grad u_\eps^\frac{p}{2}|^2 + K_2
	\end{align*}
	for all $t > 0$ and $\eps \in (0,1)$ with $K_2 \defs  m^\frac{2}{n} \max(1, m^{p})K_1$ and $m$ as defined in (\ref{approx_mass_conservation}). If we now apply this combined with Young's inequality to (\ref{eq:u_test}), we further see that
	\begin{align*}
		\frac{1}{p(p-1)} \frac{\d}{\d t} \int_\Omega u_\eps^p 
		\leq& -\frac{4}{p^2}\int_\Omega |\grad u_\eps^\frac{p}{2}|^2 + \frac{2}{p^2} \frac{1}{K_2} \int_\Omega u_\eps^{p+\frac{2}{n}} + K_3 \int_\Omega v_\eps^{1+\frac{np}{2}} \\ \leq& -\frac{2}{p^2} \int_\Omega |\grad u_\eps^\frac{p}{2}|^2 + K_3 \int_\Omega v_\eps^{1+\frac{np}{2}} + \frac{2}{p^2} \numberthis \label{eq:u_test_2}
	\end{align*}
	for all $t > 0$ and $\eps \in (0,1)$ with $K_3 \defs \frac{|\chi|}{p}(\frac{2}{|\chi| p K_2})^{-\frac{np}{2}}$.
	\\[0.5em]
	The interpolation property (\ref{eq:u_interpol_2}) from \Cref{lemma:u_interpol} then further gives us $K_4 > 0$ such that
	\begin{align*}
	\left( \int_\Omega u_\eps^{p} \right)^{1+\frac{2}{n(p-1)}} 
	&\leq  K_4m^\frac{2p}{n(p-1)}\int_\Omega |\grad u_\eps^\frac{p}{2}|^2 + K_4 m^{p + \frac{2p}{n(p-1)}}
	\\
	&\leq K_5 \int_\Omega |\grad u_\eps^\frac{p}{2}|^2 + K_5
	\end{align*}
	for all $t > 0$ and $\eps \in (0,1)$ with $K_5 \defs m^\frac{2p}{n(p-1)} \max(1, m^{p}) K_4$ and $m$ as defined in (\ref{approx_mass_conservation}). Applying this to (\ref{eq:u_test_2}) then yields
	\begin{equation}\label{eq:u_test_3}
		\frac{\d}{\d t} \int_\Omega u_\eps^p \leq -K_6 \left( \int_\Omega u_\eps^p \right)^{^{1+\frac{2}{n(p-1)}}} + K_7\int_\Omega v^{1+\frac{np}{2}}_\eps + K_7
	\end{equation}
	for all $t > 0$ and $\eps \in (0,1)$ with $K_6 \defs \frac{2}{K_5}\frac{p-1}{p}$ and $K_7 \defs \max (\,p(p-1)K_3\,,\, 4\frac{p-1}{p} \,)$ after some slight rearrangement.
	\\[0.5em]
	By now applying the Gagliardo--Nirenberg interpolation inequality (cf.\ \cite[Theorem 10.1]{FriedmanPDE}) and a standard elliptic regularity result (cf.\ \cite[Theorem 19.1]{FriedmanPDE}) combined with (\ref{eq:v_eps_leq_u_eps}), we gain $K_{8} > 0$ and $K_{9} > 0$ such that
	\begin{equation} \label{eq:v_eps_interpolation}
		\| v_\eps \|_\L{1+\frac{np}{2}} \leq K_{8} \|v_\eps\|^\alpha_{W^{2,p}(\Omega)} \|v_\eps\|^{1-\alpha}_\L{r}  \leq K_{9} \left\|\frac{u_\eps}{1+\eps u_\eps} \right\|^\alpha_\L{p} \|v_\eps\|^{1-\alpha}_\L{r} \leq K_{9} \left\|u_\eps \right\|^\alpha_\L{p} \|v_\eps\|^{1-\alpha}_\L{r} 
	\end{equation}
	for all $t > 0$ and $\eps \in (0,1)$ with 
	\[
		r \defs \frac{n(n - 1) + 2\frac{n}{p}}{2 + \frac{n}{p}}
	\]
	and
	\begin{align*}
		\alpha &= \frac{\frac{1}{r} - \frac{1}{1+\frac{np}{2}}}{ \frac{1}{r} +  \frac{2}{n} - \frac{1}{p}} 
% ======================================
		= \frac{\frac{2+\frac{n}{p}}{n(n - 1) + 2\frac{n}{p}} - \frac{1}{1+\frac{np}{2}}}{ \frac{2+\frac{n}{p}}{n(n - 1) + 2\frac{n}{p}} + \frac{2}{n} - \frac{1}{p} } 
% ======================================
		= \frac{\;\frac{(2+\frac{n}{p})(1+\frac{np}{2}) - n(n-1) - 2\frac{n}{p}}{1+\frac{np}{2}}\;}{ 2 + \frac{n}{p} + 2(n-1) + \frac{4}{p} - \frac{n(n-1)}{p} - 2\frac{n}{p^2}}\\
% ======================================
		&= \frac{1}{1+\frac{np}{2}} \left(\frac{np + n + 2 -\frac{n^2}{2} - \frac{n}{p}}{ 2n + 2\frac{n}{p} + \frac{4}{p} - \frac{n^2}{p} - 2\frac{n}{p^2}}\right) 
% ====================================== 
		= \frac{\frac{p}{2}}{1+\frac{np}{2}} \in (0,1).
	\end{align*}
	As 
	\[
		r = \frac{n(n-1) + 2\frac{n}{p}}{2+\frac{n}{p}} \geq \frac{2 + 2\frac{n}{p}}{2+\frac{n}{p}} > 1 \;\; \text{ and }  \;\; r = \frac{n(n-1) + 2\frac{n}{p}}{2+\frac{n}{p}} \leq \frac{6 + 2\frac{n}{p}}{2 + \frac{n}{p}} < 3 \leq \frac{n}{n-2}
	\]
	due to $n \in \{2,3\}$, we can now apply \Cref{lemma:starting_point} to (\ref{eq:v_eps_interpolation}), which gives us a constant $K_{10} > 0$ such that
	\begin{equation} \label{eq:v_eps_interpolation_consequence}
		\int_\Omega v_\eps^{1+\frac{np}{2}} = \| v_\eps \|^{1+\frac{np}{2}}_\L{1+\frac{np}{2}} \leq K_{9}^{1+\frac{np}{2}}K_{10}^{1+\frac{p(n-1)}{2}} \|u_\eps\|^\frac{p}{2}_\L{p} = K_{11} \left(\int_\Omega u_\eps^p\right)^\frac{1}{2}
	\end{equation}
	for all $t > 0$ and $\eps \in (0,1)$. Here, $K_{11} \defs K_{9}^{1+\frac{np}{2}}K_{10}^{1+\frac{p(n-1)}{2}}$. By Young's inequality this implies
	\[
		\int_\Omega v_\eps^{1+\frac{np}{2}} \leq \frac{K_6}{2K_7} \left(\int_\Omega u^p_\eps\right)^{1+\frac{2}{n(p-1)}} + K_{12}
	\]
	for all $t > 0$ and $\eps \in (0,1)$ with $K_{12} \defs K_{11} \left(\frac{K_6}{2 K_7 K_{11}}\right)^{-\frac{n(p-1)}{4+n(p-1)}}$as $\frac{1}{2} < 1 \leq 1 + \frac{2}{n(p-1)}$.
	\\[0.5em]
	As our final step, we now apply this result to (\ref{eq:u_test_3}) to see that
	\[
		\frac{\d}{\d t} \int_\Omega u_\eps^p \leq -\frac{K_6}{2} \left( \int_\Omega u_\eps^p \right)^{^{1+\frac{2}{n(p-1)}}} + K_7 + K_7 K_{12}.
	\]
	for all $t > 0$ and $\eps \in (0,1)$, which completes the proof.
\end{proof}
\begin{remark}
	Let us now briefly discuss why the above argument breaks down for dimensions greater than three. To do this, we start by identifying its linchpin, namely the question whether or not the $v_\eps$ integral term in (\ref{eq:u_test_3}) can be absorbed by the only term with negative sign in said same inequality. The argument employed by us to answer this question positively in two and three dimensions centrally relies on interpolation and embedding properties of certain Sobolev spaces combined with some elliptic regularity theory to gain an estimate of the form
	\[
		\int_\Omega v_\eps^{1+\frac{np}{2}} \leq C \left( \int_\Omega u_\eps^p \right)^\alpha
	\]
	for all $p \in (1,\infty)$, $\eps \in (0,1)$ and $t > 0$ with some $\alpha < 1 + \frac{2}{n(p-1)}$ and $C > 0$ (see the inequality (\ref{eq:v_eps_interpolation_consequence}) above). But as for higher dimensions the exponent of $v_\eps$ increases and the type of interpolation properties used diminish in effectiveness, using the same argument in dimension four and higher only yields the above estimate with $\alpha > 1 + \frac{2}{n(p-1)}$, which is insufficient for our approach.
\end{remark}
\noindent
While generally the more difficult case to handle and thus needing some additional restrictions on the initial data and dimension to make work, the attractive case ($\chi > 0$) can be handled in a very similar manner to the repulsive case discussed above. The main difference is that the key problematic term changes from $\int_\Omega u^p_\eps v_\eps$ to $\int_\Omega u_\eps^{p+1}$ due to $\chi$'s change of sign. While this reduces proof length as we do not first need to split up the new problematic term to separate $u_\eps$ and $v_\eps$, the $\int_\Omega u_\eps^{p+1}$ term itself is already much more difficult to handle and can therefore only be compensated for by the dissipative term in two dimensions and for small initial mass, at least when using our approach based on the Gagliardo--Nirenberg inequality.

\begin{lemma}\label{lemma:the_central_ineq_attraction}
	For each $p\in(1,\infty)$, there exist constants $C_1 \equiv C_1(p) > 0$ and $C_2 \equiv C_2(p) > 0$ such that the following holds: 
	\\[0.5em]
	If $\chi > 0$, $n = 2$ and $m \leq C_1$ with $m$ as defined in (\ref{approx_mass_conservation}), then
	\begin{equation}  \label{eq:the_central_ineq_attraction}
		\frac{\d }{\d t}\int_\Omega u_\eps^p \leq - C_2\left( \int_\Omega u_\eps^p \right)^{1+\frac{2}{n(p-1)}} + C_2
	\end{equation}
	for all $t > 0$ and $\eps \in (0,1)$.
\end{lemma}
\begin{proof}
Fix $p \in (1,\infty)$. Similar to the argument in \Cref{lemma:the_central_ineq}, we start again by testing the first equation in (\ref{approx_system}) with $u_\eps^{p-1}$ to gain	
\begin{align*}
	\frac{1}{p(p-1)} \frac{\d}{\d t} \int_\Omega u_\eps^p =& -\frac{4}{p^2}\int_\Omega |\grad u_\eps^\frac{p}{2}|^2 + \frac{|\chi|}{p}\int_\Omega \grad u^{p}_\eps \cdot \grad v_\eps \\
	=&-\frac{4}{p^2}\int_\Omega |\grad u_\eps^\frac{p}{2}|^2 - \frac{|\chi|}{p}\int_\Omega u^p_\eps \laplace v_\eps \\
	=&-\frac{4}{p^2}\int_\Omega |\grad u_\eps^\frac{p}{2}|^2 - \frac{|\chi|}{p}\int_\Omega u^p_\eps v_\eps + \frac{|\chi|}{p} \int_\Omega\frac{u_\eps^{p+1}}{1+\eps u_\eps}\\
	\leq&-\frac{4}{p^2}\int_\Omega |\grad u_\eps^\frac{p}{2}|^2 + \frac{|\chi|}{p}\int_\Omega u^{p+1}_\eps \numberthis \label{eq:u_test_attraction}
\end{align*}
for all $t > 0$ and $\eps \in (0,1)$. Due to the  interpolation result (\ref{eq:u_interpol_1}) from \Cref{lemma:u_interpol}, we further know that there exists $K_1 > 0$ such that 
\begin{equation} \label{eq:attraction_gni}
	\int_\Omega u_\eps^{p+1} \leq K_1 m \int_\Omega |\grad u_\eps^\frac{p}{2}|^2 + K_1 m^{p + 1} = K_1 m \int_\Omega |\grad u_\eps^\frac{p}{2}|^2+ K_2 
\end{equation}
for all $t > 0$ and $\eps \in (0,1)$ with $K_2 \defs K_1m^{p+1}$. Let now $C_1(p) \defs \frac{1}{|\chi|}\frac{2}{K_1 p}$. Then combining (\ref{eq:attraction_gni}) with (\ref{eq:u_test_attraction}) yields
\[
		\frac{1}{p(p-1)} \frac{\d}{\d t} \int_\Omega u_\eps^p \leq -\frac{2}{p^2}\int_\Omega |\grad u_\eps^\frac{p}{2}|^2 + \frac{|\chi|}{p}K_2
\]
for all $t > 0$ and $\eps \in (0,1)$ because $m \leq C_1$. Applying the interpolation inequality (\ref{eq:u_interpol_2}) from \Cref{lemma:u_interpol} to the above then immediately gives us our desired result.
\end{proof}

\noindent Given that we will in two dimensions actually only need the inequality (\ref{eq:the_central_ineq_attraction}) to be true for the values $p = \frac{5}{2}$ and $p = 8$ (cf.\ \Cref{lemma:u_linfty}, \Cref{lemma:ugradv_integrability}), we can now fix the relevant minimal constant $C_m > 0$ for these two cases. Note that this is exactly the constant $C_m$ mentioned in \Cref{theorem:main}.

\begin{corollary}\label{corollary:the_central_ineq_attraction}
	There exists a minimal $C_m > 0$ such that, if $\chi > 0$, $n = 2$ and $m\leq C_m$, the inequality  (\ref{eq:the_central_ineq_attraction}) holds for $p \in \{\frac{5}{2}, 8\}$ and all $t > 0$, $\eps \in (0,1)$.
\end{corollary}

\noindent 
As the differential inequalities (\ref{eq:the_central_ineq}) and (\ref{eq:the_central_ineq_attraction}) are essential for all further arguments (either to ensure sufficient regularity from a time $t_0 > 0$ onward or to ensure uniform continuity in $t = 0$), we will from now on always assume one of (\ref{scenario:2d+repulsion})--(\ref{scenario:3d+repulsion}) to be true (with the constant $C_m$ in (\ref{scenario:2d+attraction}) chosen exactly as in the corollary above). We do this, as this means that either \Cref{lemma:the_central_ineq} or \Cref{corollary:the_central_ineq_attraction} will be always usable from now on.

\section{Initial data independent a priori estimates for $u_\eps$ and $v_\eps$ on $\Omega\times(t_0,\infty)$ for all $t_0 > 0$}

Given that we have now established the central differential inequalities (\ref{eq:the_central_ineq}) and (\ref{eq:the_central_ineq_attraction}) in all of the cases (\ref{scenario:2d+repulsion})--(\ref{scenario:3d+repulsion}) and for all relevant values of $p$ used in this section, we will now prove certain $u_{0,\eps}$-independent smoothing properties of the approximate solutions. That is, we will essentially show that, from each $t_0 > 0$ onward, both solution components are bounded in sufficiently good function spaces independent of the initial data ${u_{0,\eps}}$.
\\[0.5em]
It will be these bounds combined with the compact embedding properties of said function spaces that will then allow us in the following section to gain a null sequence $(\eps_j)_{j\in\N}$, along of which the approximate solutions converge to a tuple of functions $(u, v)$, which serves as a candidate for our actual solution. %The strong convergence properties provided to us by this construction will then be used to show that the solution properties retained by the limit process are still sufficient to argue that $(u,v)$ is a classical solution to (\ref{problem}) with Neumann boundary conditions.
\\[0.5em]
As the first step in the bootstrap argument leading us toward this goal, we begin by extracting the following $L^\infty(\Omega)$ boundedness property from \Cref{lemma:the_central_ineq} or \Cref{corollary:the_central_ineq_attraction}.
\begin{lemma} \label{lemma:u_linfty}
	For each $t_0 > 0$, there exists a constant $C \equiv C(t_0) > 0$ such that
	\[
		\|u_\eps(\cdot, t)\|_\L{\infty} \leq C
	\]
	for all $t > t_0$ and $\eps \in (0,1)$. 
\end{lemma}
\begin{proof}
	Fix $t_0 > 0$. According to \Cref{lemma:the_central_ineq} or \Cref{corollary:the_central_ineq_attraction}, there exists $K_1 > 0$ such that
	\[
		\frac{\d}{\d t}\int_\Omega u_\eps^8 \leq -K_1 \left( \int_\Omega u_\eps^8 \right)^{1+\frac{2}{7n}} + K_1 
	\]
	for all $t > 0$ and $\eps \in (0,1)$, which by \Cref{lemma:ode_bound} gives us $K_2 > 0$ such that
	\[
		\int_\Omega u^8_\eps(\cdot, t) \leq K_2 t^{-\frac{7}{2}n} + K_2 \leq K_2 \left(\tfrac{1}{2}t_0\right)^{-\frac{7}{2}n} + K_2 \sfed K_3
	\]
	for all $t > \frac{1}{2}t_0$ and $\eps \in (0,1)$. Due to the Hölder inequality and standard elliptic regularity theory (cf.\ \cite[Theorem 19.1]{FriedmanPDE}) 
	combined with (\ref{eq:v_eps_leq_u_eps}), this then implies that there exists $K_4 > 0$ such that
	\begin{align*}
		\|u_\eps \grad v_\eps\|_\L{4} &\leq \|u_\eps \|_\L{8}\|\grad v_\eps\|_\L{8} \leq \|u_\eps \|_\L{8}\|v_\eps\|_{W^{2,8}(\Omega)} \\
		&\leq K_4\|u_\eps \|_\L{8}\left\|\frac{u_\eps}{1+\eps u_\eps} \right\|_\L{8} \leq K_4\|u_\eps \|^2_\L{8} \leq K_4 \sqrt[4]{K_3} \sfed K_5
	\end{align*}
	for all $t > \frac{1}{2}t_0$ and $\eps \in (0,1)$. We can now use the variation-of-constants representation of $u_\eps$ on $(t - \frac{1}{2}t_0, t)$ for all $t > t_0$ combined with  well-known smoothness properties (cf.\  \cite[Lemma 1.3]{WinklerSemigroupRegularity}) of the semigroup $(e^{t\laplace})_{t > 0}$ to gain $K_6 > 0$ such that
	\begin{align*}
		\|u_\eps(\cdot, t)\|_\L{\infty} 
		&= \left\| \; e^{\frac{1}{2}t_0\laplace} u_\eps(\cdot, t-\tfrac{1}{2}t_0) + \chi\int_{t-\frac{1}{2}t_0}^t e^{(t-s)\laplace} \div (u_\eps(\cdot, s) \grad v_\eps(\cdot, s)) \d s \; \right\|_\L{\infty} \\
		&\leq K_6\|u_\eps(\cdot, t - \tfrac{1}{2}t_0)\|_\L{8} \left(1 + (\tfrac{1}{2}t_0)^{-\frac{n}{16}}\right) + K_6\int_{t - \frac{1}{2}t_0}^{t} \left(1 + (t-s)^{-\frac{1}{2} - \frac{n}{8}}\right)\|u_\eps\grad v_\eps\|_\L{4}  \d s\\
		&\leq \sqrt[8]{K_3} K_6 \left(1 + (\tfrac{1}{2}t_0)^{-\frac{n}{16}}\right) + K_5 K_6\int_{t - \frac{1}{2}t_0}^{t} \left(1 + (t-s)^{-\frac{1}{2} - \frac{n}{8}} \right)  \d s \\
		&=  \sqrt[8]{K_3} K_6 \left(1 + (\tfrac{1}{2}t_0)^{-\frac{n}{16}}\right) + K_5 K_6 \left(\frac{1}{2}t_0 + \frac{1}{\frac{1}{2} - \frac{n}{8}}(\tfrac{1}{2}t_0)^{\frac{1}{2} - \frac{n}{8}}\right)
	\end{align*}
	for all $t > t_0$ and $\eps \in (0,1)$ because $-\frac{1}{2} -\frac{n}{8} > -1$ as $n < 4$. This completes the proof.
\end{proof}

\noindent
We can now use this result to gain a uniform, global $C^{1+\alpha}(\overline{\Omega})$-type bound for $v_\eps$ due to standard elliptic regularity theory and embedding properties of Sobolev spaces into Hölder spaces. This in turn then allows us to apply parabolic regularity theory from \cite{PorzioVespriHoelder} to achieve $C^{\alpha, \frac{\alpha}{2}}(\overline{\Omega}\times[t_0,t_1])$-type bounds for $u_\eps$.

\begin{lemma}
	\label{lemma:uv_hoelder_1}
	For each $t_0 > 0$, there exist $C \equiv C(t_0) > 0$ and $\alpha \equiv \alpha(t_0) \in (0,1)$ such that
	\begin{equation}\label{eq:u_c1plusalpha}
		\|u_\eps\|_{C^{\alpha, \frac{\alpha}{2}}(\overline{\Omega}\times[t_0, t_1])} \leq C 
	\end{equation}
	for all $t_1 > t_0$, $\eps \in (0,1)$ and 
	\begin{equation}\label{eq:v_c1plusalpha}
		\|v_\eps(\cdot, t)\|_{C^{1+\alpha}(\overline{\Omega})} \leq C
	\end{equation}
	for all $t > t_0$, $\eps \in (0,1)$.
\end{lemma}
\begin{proof}
	Fix $t_0 > 0$. \\[0.5em]
	The inequality (\ref{eq:v_c1plusalpha}) is a straightforward consequence of standard elliptic regularity theory (cf.\ \cite[Theorem~19.1]{FriedmanPDE}) combined with (\ref{eq:v_eps_leq_u_eps}), the embedding properties of the Sobolev spaces $W^{2,p}(\Omega)$ into Hölder spaces $C^{1+\beta}(\overline{\Omega})$ for sufficiently large values of $p$ in relation to $\beta$ (cf.\ \cite[Theorem 2.72]{EllipticFunctionSpaces}), the Hölder inequality and \Cref{lemma:u_linfty}.
	\\[0.5em]
	After this straightforward application of elliptic regularity theory, we will now transition to its parabolic counterpart found in \cite[Theorem 1.3]{PorzioVespriHoelder} to gain (\ref{eq:u_c1plusalpha}). For this, we first fix $K_1 > 0$ such that 
	\begin{equation} \label{eq:bound_for_porzio_vespri}
		\|u_\eps(\cdot, t)\|_\L{\infty} \leq K_1  \stext{and} \|\grad v_\eps(\cdot, t)\|_\L{\infty} \leq K_1 
	\end{equation}
	for all $t \geq \tfrac{1}{2}t_0$ and $\eps \in (0,1)$ according to \Cref{lemma:u_linfty} and (\ref{eq:v_c1plusalpha}). 
	\\[0.5em]
	In the notation of reference \cite{PorzioVespriHoelder}, the first equation of (\ref{approx_system}) considered in isolation can be then be written as \[
		u_{\eps t} - \div a(x,t,u_\eps, \grad u_\eps) = b(x,t,u_\eps,\grad u_\eps)
	\] with $a(x,t,y,z) \defs z - \chi u_\eps(x,t)\grad v_\eps(x, t)$ for all $(x,t,y,z) \in \Omega\times[t-\tfrac{1}{2}t_0, t+1]\times\R\times\R^n$ and $b \equiv 0$ on any interval $[t-\tfrac{1}{2}t_0, t+1]$ for all $t \geq t_0$. Due to (\ref{eq:bound_for_porzio_vespri}), it is easy to see that $a$ and $b$ have all necessary structure conditions with constants and parameters only depending on $t_0$ and $K_1$. An application of Theorem 1.3 from \cite{PorzioVespriHoelder} on each interval $[t-\tfrac{1}{2}t_0, t+1]$, $t \geq t_0$, therefore yields constants $\alpha \in (0,1)$ and $K_2 > 0$, only dependent on $t_0$ and $K_1$, such that
	\begin{equation} \label{eq:u_hoelder_1_close}
		|u_\eps(x,t) - u_\eps(y,s)| \leq K_2 (|x-y|^\alpha + |t-s|^\frac{\alpha}{2} )
	\end{equation}
	for all $\eps \in (0,1)$, $x,y \in \overline{\Omega}$ and $t,s \in [t_0, \infty)$ with $|t-s| < 1$. Conversely,
	\begin{equation} \label{eq:u_hoelder_1_apart}
		|u_\eps(x,t) - u_\eps(y,s)| \leq 2 K_1 \leq 2 K_1 ( |x-y|^\alpha + |t-s|^\frac{\alpha}{2} )
	\end{equation}
	for all $\eps \in (0,1), x,y \in \overline{\Omega}$ and $t,s\in[t_0,\infty)$  with $|t-s| \geq 1$. Combined with (\ref{eq:bound_for_porzio_vespri}), the inequalities (\ref{eq:u_hoelder_1_close}) and (\ref{eq:u_hoelder_1_apart}) yield our desired result.
\end{proof} \noindent
Given that the critical source term in the second equation in (\ref{approx_system}) is in fact not $u_\eps$ but $\frac{u_\eps}{1+\eps u_\eps}$, we now derive the following corollary translating the above results to said source term.
\begin{corollary}\label{corollary:u_hoelder_1}
	For each $t_0 > 0$, there exist $C \equiv C(t_0) > 0$ and $\alpha \equiv \alpha(t_0) \in (0,1)$ such that
	\[
		\left\|\frac{u_\eps}{1+\eps u_\eps}\right\|_{C^{\alpha, \frac{\alpha}{2}}(\overline{\Omega}\times[t_0, t_1])} \leq C
	\]
	for all $t_1 > t_0$ and $\eps \in (0,1)$.
\end{corollary}
\begin{proof}
	This follows directly from (\ref{eq:u_c1plusalpha}) in \Cref{lemma:uv_hoelder_1} due to the fact that
	\[
		\left| \frac{u_\eps(x,t)}{1+\eps u_\eps(x,t)} \right| \leq |u_\eps(x,t)| 
	\] 
	and
	\[
		\left|\frac{u_\eps(x,t)}{1+\eps u_\eps(x,t)} - \frac{u_\eps(y,s)}{1+\eps u_\eps(y,s)}\right| = \frac{|u_\eps(x,t)-u_\eps(y,s)|}{(1+\eps u_\eps(x,t))(1+\eps u_\eps(y,s))} \leq |u_\eps(x,t)-u_\eps(y,s)|
	\]
	for all $x,y \in \overline{\Omega}$ and $t,s \in [0,\infty)$.
\end{proof}\noindent
While the Hölder bounds derived above are already quite strong and something similar to the following lemma could likely be derived from them by more abstract means, we will now give a very short argument, which provides us with a crucial bound for the gradients of the functions $u_\eps$. The argument is based on a straightforward testing procedure for the first equation in (\ref{approx_system}) with $u_\eps$. We do this to later be able to translate certain weak solution properties from the approximate solutions to the solution candidates constructed in the following section.
\begin{lemma}\label{lemma:grad_u_bound}
	For each $t_1> t_0 > 0$, there exists a constant $C \equiv C(t_0, t_1) > 0$ such that
	\[
	\|\grad u_\eps\|_{L^2(\Omega\times(t_0, t_1))} \leq C.
	\]
	for all $\eps \in (0,1)$.
\end{lemma}
\begin{proof}
	Fix $t_0 > 0$.
	According to \Cref{lemma:uv_hoelder_1}, there exists $K > 0$ such that
	\begin{equation}\label{eq:baseline_bounds_for_grad}
		\|u_\eps(\cdot, t)\|_\L{\infty} \leq K \stext{and} \|\grad v_\eps(\cdot, t)\|_\L{\infty} \leq K
	\end{equation}
	for all $t \geq t_0$ and $\eps \in (0,1)$.
	Testing the first equation in (\ref{approx_system}) with $u_\eps$ then directly yields
	\begin{align*}
	\frac{1}{2} \frac{\d}{\d t} \int_\Omega u_\eps^2 =& -\int_\Omega |\grad u_\eps|^2 + \chi \int_\Omega u_\eps \grad u_\eps \cdot \grad v_\eps \\
	&\leq -\frac{1}{2}\int_\Omega |\grad u_\eps|^2 + \frac{1}{2}|\chi|^2\int_\Omega u_\eps^2 |\grad v_\eps|^2\\
	&\leq -\frac{1}{2}\int_\Omega |\grad u_\eps|^2 + \frac{1}{2}|\chi|^2|\Omega| K^4
	\end{align*}
	for all $t > t_0$ and $\eps \in (0,1)$. A straightforward time integration of the above combined with (\ref{eq:baseline_bounds_for_grad}) completes the proof.
\end{proof}\noindent
As the last important a priori bound of this section, we again use standard elliptic regularity theory to provide a stronger uniform Hölder bound for the approximate solution components $v_\eps$. This will allow us to later use the fact that all approximate solution components $v_\eps$ are classical solutions of the second equation in (\ref{approx_system}) to argue that their limits are in fact classical solutions of the second equation in (\ref{problem}). One key idea here is to apply said elliptic regularity theory not only to the functions $v_\eps$ themselves but also to the difference functions $v_\eps(\cdot, t) - v_\eps(\cdot, s)$ and use the already established parabolic Hölder bounds for $u_\eps$ to gain time Hölder bounds from regularity theory that is originally only interested in the space variable.
\begin{lemma} \label{lemma:v_hoelder_2}
	For each $t_0 > 0$, there exist $C \equiv C(t_0) > 0$ and $\alpha \equiv \alpha(t_0) \in (0,1)$ such that
	\begin{equation*}
	\|v_\eps\|_{C^{\frac{\alpha}{2}}([t_0,t_1]; C^{2+\alpha}(\overline{\Omega}))} \leq C 
	\end{equation*}
	for all $t_1 > t_0$ and $\eps \in (0,1)$.
\end{lemma}
\begin{proof}
	Fix $t_0 > 0$. Due to \Cref{corollary:u_hoelder_1}, we can then fix $K_1 > 0$ and $\alpha \in (0,1)$ such that
	\[
		\left\|\frac{u_\eps}{1+\eps u_\eps}\right\|_{C^{\frac{\alpha}{2}} ([t_0, t_1]; C^\alpha(\overline{\Omega}))} \leq K_1
	\]
	for all $t_1 > t_0$ and $\eps \in (0,1)$.
	\\[0.5em]
	If we now employ the elliptic regularity theory found in Theorem 3.1 of \cite[p.135]{LUElliptic} combined with (\ref{eq:v_eps_leq_u_eps}), we immediately gain $K_2 > 0$ such that
	\begin{align*}
		\|v_\eps(\cdot, t)\|_{C^{2+\alpha}(\overline{\Omega})} &\leq K_2 \left\|\frac{u_\eps(\cdot,t)}{1+\eps u_\eps(\cdot, t)}\right\|_{C^\alpha(\overline{\Omega})} \leq K_1 K_2 \label{eq:v_hoelder_1} \numberthis 
	\end{align*}
	for all $t > t_0$ and $\eps \in (0,1)$. Note further that
	\[
		0 = \laplace (v_\eps(x, t) - v_\eps(x, s)) - (v_\eps(x, t) - v_\eps(x, s)) + \frac{u_\eps(\cdot, t)}{1+\eps u_\eps(\cdot, t)} - \frac{u_\eps(\cdot, s)}{1+\eps u_\eps(\cdot, s)} \;\;\;\; \text{ for all } x \in \Omega
	\]
	and
	\[	
		\grad (v_\eps(x, t) - v_\eps(x, s)) \cdot \nu = 0 \;\;\;\; \text{ for all } x \in \partial\Omega
	\]
	for all $s,t \in [t_0,\infty)$ and $\eps \in (0,1)$. This makes the difference function $v_\eps(\cdot, t) - v_\eps(\cdot, s)$ accessible to the same elliptic regularity theory from \cite{LUElliptic} and a similar argument as the one used to derive (\ref{eq:v_eps_leq_u_eps}). Therefore, we can find $K_3 > 0$ such that
	\begin{align*} 
		\|v_\eps(\cdot, t) - v_\eps(\cdot, s)\|_{C^{2+\alpha}(\overline{\Omega})} &\leq K_3\left\|\frac{u_\eps(\cdot, t)}{1+\eps u_\eps(\cdot, t)} - \frac{u_\eps(\cdot, s)}{1+\eps u_\eps(\cdot, s)}\right\|_{C^{\alpha}(\overline{\Omega})} 
		\leq K_1 K_3 (t-s)^\frac{\alpha}{2} \numberthis \label{eq:v_hoelder_2}
	\end{align*}
	for all $t,s \in [t_0,\infty)$ and $\eps \in (0,1)$. Combining (\ref{eq:v_hoelder_1}) with (\ref{eq:v_hoelder_2}) then yields our desired result.
\end{proof}

\section{Construction of our solution candidates $(u,v)$ as limits of the approximate solutions}
The target of this section will be to construct solution candidates $(u,v)$, which solve (\ref{problem}) and adhere to the boundary conditions (\ref{boundary_conditions}), by using the a priori estimates from the previous section. As said estimates are all gained in a fashion that was initial data independent by necessity, the construction in this section will by itself not provide us with any type of continuity in $t = 0$. This issue will instead be addressed in the section directly following.
\\[0.5em]
We begin by using the compact embedding properties of various function spaces combined with the bounds derived in the previous section to construct our solution candidates as limits of the approximate solutions by multiple subsequence extraction and diagonal sequence arguments.
\begin{lemma}\label{lemma:construction}
	There exist a null sequence $(\eps_j)_{j\in\N}\subseteq (0,1)$, functions $u,v: \overline{\Omega}\times(0,\infty) \rightarrow [0,\infty)$ and, for each $t_1 > t_0 > 0$, there exists $\alpha \equiv \alpha(t_0,t_1) \in (0,1)$ such that
	\begin{align*}
		u_\eps &\rightarrow u && \text{ in } C^{\alpha, \frac{\alpha}{2}}(\overline{\Omega}\times[t_0,t_1]),\numberthis \label{eq:u_hoelder_approx} \\
		u_\eps &\rightharpoonup u && \text{ in } L^2((t_0,t_1);W^{1,2}(\Omega)) \text{ and } \numberthis \label{eq:sobolev_approx} \\
		v_\eps &\rightarrow v && \text{ in } C^\frac{\alpha}{2}([t_0,t_1]; C^{2+\alpha}(\overline{\Omega})) \numberthis \label{eq:v_hoelder_approx}
	\end{align*}
	as $\eps = \eps_j \searrow 0$.
\end{lemma}
\begin{proof}
	Due to \Cref{lemma:uv_hoelder_1}, there exist $\beta_k \in (0,1)$ and $K_k > 0$ for $k \in \N$ such that 
	\[
		\|u_\eps\|_{C^{\beta_k, \frac{\beta_k}{2}}(\overline{\Omega}\times[\frac{1}{k}, k])} \leq K_k \;\;\;\; \text{ for all } k \in \N.
	\]
	Due to the compact embedding properties of Hölder spaces this implies that, for each $k \in \N$, there exist $\alpha_k \in (0,1)$, functions $u_k: \overline{\Omega}\times [\frac{1}{k}, k] \rightarrow [0,\infty)$ and sequences $(\eps_j^{(k)})_{j\in\N} \subseteq (0,1)$ such that $\eps^{(k)}_j \searrow 0$ as $j \rightarrow \infty$,  $(\eps_j^{(k+1)})_{j\in\N}$ is a subsequence of $(\eps_j^{(k)})_{j\in\N}$ and 
	\[
		u_\eps \rightarrow u_k \text{ in } C^{\alpha_k, \frac{\alpha_k}{2}}(\overline{\Omega}\times[\tfrac{1}{k}, k]) \text{ as } \eps = \eps^{(k)}_j \searrow 0.
	\]
	These sequences are gained by multiple subsequence extraction arguments. Due to the uniqueness of pointwise limits, this implies that $u_{k+1}\vert_{[\frac{1}{k}, k]} = u_k$ for all $k \in \N$. Therefore, we can define $u$ as 
	\begin{equation}\label{eq:u_construction}
		u(x,t) \defs u_k(x,t) \;\;\;\; \text{ for all } (x,t) \in \overline{\Omega}\times(0,\infty) \text{ and some } k\in\N \text{ such that } t \in [\tfrac{1}{k}, k] 
	\end{equation}
	in well-defined fashion.
	If we now set $\eps_j \defs \eps^{(j)}_j$, $j \in \N$, we gain our first desired convergence property (\ref{eq:u_hoelder_approx}) using a standard diagonal sequence argument.
	\\[0.5em]
	By essentially the same argument as above, but this time derived from the bound established in \Cref{lemma:v_hoelder_2}, we can now construct a function $v: \overline{\Omega}\times(0,\infty) \rightarrow [0,\infty)$ and extract another subsequence from the previous one, along of which the convergence property (\ref{eq:v_hoelder_approx}) holds.
	\\[0.5em]
	Finally, \Cref{lemma:u_linfty} and \Cref{lemma:grad_u_bound} allow us to extract a final subsequence with the property (\ref{eq:sobolev_approx}) by using compactness properties (relative to the weak topology) of bounded sets in the spaces $L^2((\frac{1}{k},k);W^{1,2}(\Omega))$, $k\in\N$, combined with a similar diagonal sequence argument. Note hereby that the convergence property (\ref{eq:u_hoelder_approx}) derived above ensures that all these weak limits coincide with the limit function $u$ from (\ref{eq:u_construction}).
\end{proof}
\noindent
These convergence properties are now immediately good enough to ensure that $v$ is a $C^{2,0}(\overline{\Omega}\times(0,\infty))$ classical solution of the second equation in (\ref{problem}) with Neumann boundary conditions, but are insufficient to derive the same solution property for $u$ and the first equation in (\ref{problem}). Instead, we will first show that $u$ is a weak solution compatible with well-known parabolic theory from \cite{LSU} and then use said theory to argue that it was already classical. This is done by using the fact that such weak solutions are indeed unique and the fact that due to the high regularity of $v$, the associated classical problem always has a sufficiently smooth solution. \\[0.5em]
We chose this approach instead of first providing stronger a priori estimates with similar parabolic regularity theory for the approximate solutions, which would yield stronger convergence properties and therefore that $u$ is immediately a classical solution, because it is rather more involved to gain the necessary uniform bounds in an $\eps$ independent fashion from said parabolic theory as opposed to the mere fact that the weak solutions considered here are indeed classical.

\begin{lemma}\label{lemma:uv_are_actually_solutions}
	The functions $u,v$ constructed in \Cref{lemma:construction} have the properties
	\begin{align*}
		u \in C^{2,1}(\overline{\Omega}\times(0,\infty)), \;\;\;\;
		v \in C^{2,0}(\overline{\Omega}\times(0,\infty))
	\end{align*}
	and satisfy (\ref{problem}) on $\Omega\times(0,\infty)$ and the boundary conditions (\ref{boundary_conditions}).
\end{lemma}
\begin{proof}
	Let $u,v$ be the functions and $(\eps_j)_{j\in\N}$ be the sequence constructed in \Cref{lemma:construction}.
	\\[0.5em]
	That $v \in C^{2,0}(\overline{\Omega}\times(0,\infty))$ is immediately obvious from (\ref{eq:v_hoelder_approx}). Further, all $v_\eps$ satisfy the second equation in (\ref{approx_system}), which is apart from one term very similar to the second equation in (\ref{problem}), and have Neumann boundary conditions. Therefore, the same convergence property directly allows us to conclude that $v$ in fact satisfies the second equation in (\ref{problem}) on $\Omega\times(0,\infty)$ and has Neumann boundary conditions due to the convergence property (\ref{eq:u_hoelder_approx}) ensuring pointwise convergence of $u_\eps$ to $u$ on $\Omega\times(0,\infty)$ and the fact that $\frac{x}{1+\eps x} \rightarrow x$ as $\eps \searrow 0$ for all $x \in [0,\infty)$.
	\\[0.5em]
	For the function $u$, we first argue that it is in fact a weak solution of the first equation in (\ref{problem}) on $\Omega\times(t_0,\infty)$ for all $t_0 > 0$ in the sense that $u \in C(\overline{\Omega}\times[t_0, \infty))\cap L_\loc^2([t_0, \infty);W^{1,2}(\Omega))$ and $u$ fulfills 
	\begin{equation}\label{eq:weak_solution_formulation}
		\int_{t_0}^\infty \int_\Omega u \phi_t + \int_\Omega u(\cdot, t_0)\phi(\cdot, t_0) = \int_{t_0}^\infty \int_\Omega \grad u \cdot \grad \phi - \chi \int_{t_0}^\infty \int_\Omega u \grad v \cdot \grad \phi
	\end{equation}
	for all $\phi \in C^\infty(\overline{\Omega}\times[t_0, \infty))$ with compact support and $t_0 > 0$. For this purpose, we now fix $t_0 > 0$ and then notice that the convergence properties (\ref{eq:u_hoelder_approx}) and (\ref{eq:sobolev_approx}) already ensure that $u \in C(\overline{\Omega}\times[t_0, \infty))\cap L_\loc^2([t_0, \infty);W^{1,2}(\Omega))$. We observe further that the approximate solutions $u_\eps$ already are weak solutions of the above type as is easily seen by partial integration. Thus, we now only need to prove that this solution property survives taking the limit $\eps = \eps_j \searrow 0$. For the first, second and fourth integral term in (\ref{eq:weak_solution_formulation}), this is immediately ensured by the convergence properties (\ref{eq:u_hoelder_approx}) and (\ref{eq:v_hoelder_approx}). The remaining third integral term converges as desired due to (\ref{eq:sobolev_approx}).
	\\[0.5em]
	As our final step of this proof, we will now rely on some well-known parabolic regularity results from \cite{LSU} to show that the weak solution formulation above combined with the already known regularity properties for $u$ and $v$ already imply that $u$ is a classical solution to the corresponding partial differential equation. As (\ref{eq:v_hoelder_approx}) already implies that the coefficients of the operator $\mathcal{L}(u,\grad u, \laplace u) = -\laplace u + \chi \grad u \cdot \grad v + \chi u \laplace v$ are in fact elements of $C^{\alpha, \frac{\alpha}{2}}(\overline{\Omega} \times [t_0, t_1])$ for all $t_1 > t_0$ and some $\alpha \equiv \alpha(t_0,t_1) \in (0,1)$, a combination of Theorem 5.1 from \cite[p.170]{LSU}, which ensures uniqueness of such weak solutions, Theorem 5.3 from \cite[p.320]{LSU}, which ensures the existence of higher regularity classical solutions to this problem, and a standard cut-off argument, which helps deal with the missing initial data regularity, then grants us the remaining desired results.
\end{proof}

\section{$\Mp$-valued continuity of $u$ in $t = 0$} \label{section:inital_data_regularity}

After having now constructed our solution candidate $(u,v)$ in \Cref{lemma:construction} and shown that it is already a classical solution to (\ref{problem}) on $\Omega\times(0,\infty)$ with Neumann boundary conditions, we now only need to argue that said solution candidate is related to the initial data in a sensible way (as it would be very easy to construct solutions that are not, e.g.\ certain constant ones). 
\\[0.5em]
Because our initial data are in many cases only measure-valued, the way we want to do this is to show that $u$ converges to the initial data $\idata$ in the vague topology on $\Mp$ as $t\searrow 0$, meaning that it is continuous in $t=0$ in said topology with a value of $\mu$ at $t = 0$. 
We will do this by showing that the first components of the approximate solutions in a sense already are uniformly continuous in $t = 0$ in the vague topology, which can then be easily translated to a similar continuity property for the limit function $u$. We do this by proving that terms of the form $
	\int_\Omega u_\eps(\cdot, t)\phi - \int_\Omega u_\eps(\cdot, 0)\phi =  \int_0^t \int_\Omega u_{\eps t} \phi
$ tend uniformly to $0$ as $t \searrow 0$. 
For this, we first prepare a lemma, which proves a similar uniform convergence result for the critical term occurring in our later arguments treating the $\int_0^t \int_\Omega u_{\eps t} \phi$ term. 
\\[0.5em]
Note further that it is here, where the last remaining unused assumption in (\ref{scenario:2d+repulsion})--(\ref{scenario:3d+repulsion}), namely the higher initial data regularity in the three dimensional case, comes into play. In three dimensions, it is necessary to have a better than $\L{\frac{3}{2}}$-type uniform bound for $\grad v_\eps$, which is just about not provided by \Cref{lemma:starting_point}, but available to us for slightly better initial data regularity. Everything before this point was indeed possible for measure-valued initial data.
\begin{lemma}\label{lemma:ugradv_integrability}
For each $\delta > 0$, there exists $t_0 \equiv t_0(\delta) > 0$ such that
\[
	\int_0^t \|u_\eps(\cdot, s) \grad v_\eps(\cdot, s)\|_\L{1} \d s \leq \delta
\] 
for all $t \in (0,t_0)$ and $\eps \in (0,1)$.
\end{lemma}
\begin{proof}
	Let us first treat the simpler case of $n = 2$. Then a combination of the Hölder inequality and \Cref{lemma:starting_point} gives us a constant $K_1 > 0$ such that
	\[
		\|u_\eps(\cdot, t) \grad v_\eps(\cdot, t)\|_\L{1} \leq K_1\|u_\eps(\cdot, t)\|_\L{\frac{5}{2}} = K_1 \left(\int_\Omega u_\eps^{5/2}(\cdot,t) \right)^{\frac{2}{5}},
	\]
	for all $t > 0$ and $\eps \in (0,1)$ because $\frac{5}{3} \leq 2 = \frac{n}{n-1}$. By use of \Cref{lemma:ode_bound}, \Cref{lemma:the_central_ineq} (for the case $\chi < 0$) and \Cref{corollary:the_central_ineq_attraction} (for the case $\chi > 0$), this can be improved to
	\[
		\|u_\eps(\cdot, t) \grad v_\eps(\cdot, t)\|_\L{1} \leq K_2 t^{\frac{2}{5}\frac{1}{1-\alpha}} + K_2 = K_2 t^{-\frac{3}{5}} + K_2
	\]
	with
	\[
		\alpha = 1 + \frac{2}{n(\tfrac{5}{2}-1)}= 1 + \frac{2}{3}
	\]
	for all $t > 0$, $\eps \in (0,1)$ and $K_2 > 0$ given by the prior results. As time integration then yields
	\[
			\int_0^t \|u_\eps(\cdot, s) \grad v_\eps(\cdot, s)\|_\L{1} \d s \leq \frac{5K_2}{2} t^\frac{2}{5} + K_2 t
	\]
	for all $t > 0$ and $\eps \in (0,1)$, this already implies our desired result by setting $t_0 \defs  \min(\frac{\delta}{2}\frac{1}{K_2}, (\frac{\delta}{2}\frac{2}{5K_2})^\frac{5}{2})$.
	\\[0.5em]
	We now treat the more challenging case of $n = 3$. Given that we are then in scenario (\ref{scenario:3d+repulsion}), we can assume that $\mu = f$ with $f \in L^p(\Omega)$ for some $p > 1$ and therefore there exists $K_3 > 0$ such that
	\[
		\|u_{0,\eps}\|_\L{p} \leq K_3
	\]
	for all $\eps \in (0,1)$ according to (\ref{eq:better_approx}). Due to standard embedding properties of Lebesgue spaces we can assume $p < 3$ without loss of generality. Integration of (\ref{eq:the_central_ineq}) from \Cref{lemma:the_central_ineq} then immediately yields $K_4 > 0$ such that
	\[
		\|u_\eps(\cdot,t)\|_\L{p} \leq K_4
	\]
	for all $t\in (0,1)$ and $\eps > 0$, which due to standard elliptic regularity theory (cf.\ \cite[Lemma 19.1]{FriedmanPDE}) combined with (\ref{eq:v_eps_leq_u_eps}) and Sobolev embedding theorems (cf.\ \cite{EllipticFunctionSpaces}) further gives us $K_5 > 0$ and $r \in (\frac{3}{2}, \frac{3p}{3-p})$ such that
	\[
		\|\grad v_\eps(\cdot, t)\|_\L{r} \leq K_5
	\]
	for all $t \in (0,1)$ and $\eps > 0$. By a similar application of the Hölder inequality as in the two-dimensional case, we then gain that
	\[
		\|u_\eps(\cdot, t) \grad v_\eps(\cdot, t)\|_\L{1} \leq K_5 \left(\int_\Omega u^q_\eps(\cdot,t) \right)^\frac{1}{q}
	\]
	with $q = \frac{r}{r - 1} < 3$ for all $t \in (0,1)$ and $\eps \in (0,1)$. Again by application of \Cref{lemma:the_central_ineq} combined with \Cref{lemma:ode_bound}, we gain $K_6 > 0$ such that
	\begin{equation}\label{eq:integration_2}
		\|u_\eps(\cdot,t)\grad v_\eps(\cdot,t)\|_\L{1} \leq K_6 t^{\frac{1}{q}\frac{1}{1-\beta}} + K_6 
	\end{equation}
	with 
	\[
		\beta = 1 + \frac{2}{n(q-1)}
	\]
	for all $t \in (0,1)$ and $\eps \in (0,1)$. Because this implies that
	\[
		\frac{1}{q}\frac{1}{1-\beta} = -\frac{n}{2}\frac{q-1}{q} = \frac{3}{2} \left(\frac{1}{q} - 1\right) > \frac{3}{2} \left(\frac{1}{3} - 1\right) = -1,
	\]
	the estimate (\ref{eq:integration_2}) is sufficient to complete the proof by a similar time integration argument as used in the two-dimensional case.
\end{proof}\noindent 
Given this, we can transition to the analysis of the critical $\int_0^t \int_\Omega u_{\eps t} \phi$ term, which is now pretty straightforward.
\begin{lemma}\label{lemma:u_time_diff_integrability}
	For each $\varphi \in C^2(\overline{\Omega})$ with $\grad \varphi \cdot \nu = 0$ on $\partial\Omega$ and $\delta > 0$, there exists $t_0 \equiv t_0(\delta,\varphi) > 0$ such that
	\[
		\left| \int_0^t \int_\Omega u_{\eps t} \varphi \right| < \delta
	\] 
	for all $t \in (0,t_0)$ and $\eps \in (0,1)$.
\end{lemma}
\begin{proof}
	Let $\varphi \in C^2(\overline{\Omega})$ with $\grad \varphi \cdot \nu = 0$ on $\partial\Omega$. We then test the first equation in (\ref{approx_system}) with $\varphi$ and use partial integration to see that
	\begin{align*}
		\left| \int_0^t \int_\Omega u_{\eps t} \varphi \right| &= \left| \int_0^t\int_\Omega u_\eps \laplace \varphi + \chi \int_0^t\int_\Omega u_\eps \grad v_\eps \cdot \grad \varphi \right| \\ 
		&\leq m\|\laplace \phi\|_\L{\infty} t + |\chi| \|\grad \varphi\|_\L{\infty} \int_0^t \|u_\eps(\cdot, s) \grad v_\eps(\cdot, s)\|_\L{1} \d s
	\end{align*}
	Due to \Cref{lemma:ugradv_integrability}, this already implies our desired result.
\end{proof}\noindent
Finally, due to already proving sufficiently strong convergence properties for the approximate solutions in \Cref{lemma:construction}, we now only need to show that the uniform continuity in $t = 0$ hinted at in the previous lemma translates to our solution candidates as proper continuity in $t = 0$ in the vague topology. We do this as follows:
\begin{lemma}\label{lemma:initial_data_continuity}
	The function $u$ constructed in \Cref{lemma:construction} has the following property:
	\[
		u(\cdot, t) \rightarrow \mu \;\;\;\; \text{ in } \Mp \text{ as } t\searrow 0.
	\]
\end{lemma}
\begin{proof}
	Let $\phi \in C^2(\overline{\Omega})$ with $\grad \phi \cdot \nu = 0$ on $\partial \Omega$ and $\delta > 0$. Then due to the fundamental theorem of calculus and \Cref{lemma:u_time_diff_integrability}, there exists a time $t_0 > 0$ such that
	\[
		\left|\int_\Omega u_{\eps}(\cdot, t) \varphi - \int_\Omega u_\eps(\cdot, 0)\varphi \right| = \left| \int_0^t\int_\Omega u_{\eps t} \phi \right| \leq \frac{\delta}{3}
	\]	
	for all $\eps \in (0,1)$ and $t \in (0,t_0)$. Further by (\ref{approx_initial_data}), there exists $\eps' \in (0,1)$ such that
	\[
		\left|\int_\Omega u_{\eps'}(\cdot, 0) \varphi - \int_{\overline{\Omega}} \varphi \d \mu \right| = \left|\int_\Omega u_{0,{\eps'}} \varphi - \int_{\overline{\Omega}} \varphi \d \mu \right| \leq  \frac{\delta}{3}
	\] 
	for all $\eps \in (0,\eps')$. Due to the convergence property (\ref{eq:u_hoelder_approx}), there moreover exists $\eps(t) \in (0, \eps')$ for each $t \in (0,t_0)$ such that
	\[
		\left|\int_\Omega u(\cdot, t) \varphi - \int_\Omega u_{\eps(t)}(\cdot, t) \phi \right| \leq  \frac{\delta}{3}.
	\]
	This then implies that
	\begin{align*}
		&\left| \int_\Omega u(\cdot, t)\varphi - \int_{\overline{\Omega}} \varphi \d \mu \right| \\
		\leq& \left|\int_\Omega u(\cdot, t) \phi - \int_\Omega u_{\eps(t)}(\cdot, t) \varphi  \right| + \left|\int_\Omega u_{\eps(t)}(\cdot, t) \varphi - \int_\Omega u_{\eps(t)}(\cdot, 0)\varphi \right|  + \left|\int_\Omega u_{\eps(t)}(\cdot, 0) \varphi - \int_{\overline{\Omega}} \varphi \d \mu \right| \leq \delta
	\end{align*}
	for all $t \in (0,t_0)$.
	Due to the density of functions from $C^2(\overline{\Omega})$ with Neumann boundary conditions in $C(\overline{\Omega})$ because of the sufficiently smooth boundary of $\Omega$, this already completes the proof.
\end{proof}

\section{Proof of \Cref{theorem:main}}
Combining the final result of the two previous sections now immediately gives us the central result of this paper:
\begin{proof}[Proof of \Cref{theorem:main}]
	Let $(u,v)$ be as in \Cref{lemma:construction}. Then \Cref{lemma:uv_are_actually_solutions} and \Cref{lemma:initial_data_continuity} imply all the desired solution properties of $(u,v)$ when applied in combination.
\end{proof}

\section*{Acknowledgment} The author acknowledges the support of the \emph{Deutsche Forschungsgemeinschaft} within the scope of the project \emph{Emergence of structures and advantages in cross-diffusion systems}, project number 411007140.


\begin{thebibliography}{10}
	
	\bibitem{BauerMeasureAndIntegration}
	\textsc{Bauer, H.}:
	\newblock {\em Measure and integration theory}, volume~26 of {\em De Gruyter
		Studies in Mathematics}.
	\newblock Walter de Gruyter \& Co., Berlin, 2001.
	\newblock Translated from the German by Robert B. Burckel.
	\newblock \href {http://dx.doi.org/10.1515/9783110866209}
	{\path{doi:10.1515/9783110866209}}.
	
	\bibitem{Survey}
	\textsc{Bellomo, N.}, \textsc{Bellouquid, A.}, \textsc{Tao, Y.}, and
	\textsc{Winkler, M.}:
	\newblock {\em Toward a mathematical theory of {K}eller--{S}egel models of
		pattern formation in biological tissues}.
	\newblock Math. Models Methods Appl. Sci., 25(9):1663--1763, 2015.
	\newblock \href {http://dx.doi.org/10.1142/S021820251550044X}
	{\path{doi:10.1142/S021820251550044X}}.
	
	\bibitem{MR3411404}
	\textsc{Biler, P.} and \textsc{Zienkiewicz, J.}:
	\newblock {\em Existence of solutions for the {K}eller--{S}egel model of
		chemotaxis with measures as initial data}.
	\newblock Bull. Pol. Acad. Sci. Math., 63(1):41--51, 2015.
	\newblock \href {http://dx.doi.org/10.4064/ba63-1-6}
	{\path{doi:10.4064/ba63-1-6}}.
	
	\bibitem{BrezisStraussEllipticL1Regularity}
	\textsc{Br\'{e}zis, H.} and \textsc{Strauss, W.~A.}:
	\newblock {\em Semi-linear second-order elliptic equations in {$L^{1}$}}.
	\newblock J. Math. Soc. Japan, 25:565--590, 1973.
	\newblock \href {http://dx.doi.org/10.2969/jmsj/02540565}
	{\path{doi:10.2969/jmsj/02540565}}.
	
	\bibitem{MR2549326}
	\textsc{Cie\'{s}lak, T.}, \textsc{Lauren\c{c}ot, P.}, and
	\textsc{Morales-Rodrigo, C.}:
	\newblock {\em Global existence and convergence to steady states in a
		chemorepulsion system}.
	\newblock In {\em Parabolic and {N}avier-{S}tokes equations. {P}art 1},
	volume~81 of {\em Banach Center Publ.}, pages 105--117. Polish Acad. Sci.
	Inst. Math., Warsaw, 2008.
	\newblock \href {http://dx.doi.org/10.4064/bc81-0-7}
	{\path{doi:10.4064/bc81-0-7}}.
	
	\bibitem{EllipticFunctionSpaces}
	\textsc{Demengel, F.} and \textsc{Demengel, G.}:
	\newblock {\em Functional spaces for the theory of elliptic partial
		differential equations}.
	\newblock Universitext. Springer, London; EDP Sciences, Les Ulis, 2012.
	\newblock Translated from the 2007 French original by Reinie Ern\'{e}.
	\newblock \href {http://dx.doi.org/10.1007/978-1-4471-2807-6}
	{\path{doi:10.1007/978-1-4471-2807-6}}.
	
	\bibitem{FriedmanPDE}
	\textsc{Friedman, A.}:
	\newblock {\em Partial differential equations}.
	\newblock Holt, Rinehart and Winston, Inc., New York-Montreal, Que.-London,
	1969.
	
	\bibitem{LocalExistenceInSimilarSetting}
	\textsc{Fujie, K.}, \textsc{Winkler, M.}, and \textsc{Yokota, T.}:
	\newblock {\em Boundedness of solutions to parabolic--elliptic
		{K}eller--{S}egel systems with signal-dependent sensitivity}.
	\newblock Math. Methods Appl. Sci., 38(6):1212--1224, 2015.
	\newblock \href {http://dx.doi.org/10.1002/mma.3149}
	{\path{doi:10.1002/mma.3149}}.
	
	\bibitem{HorstmannBoundednessVsBlowup2005}
	\textsc{Horstmann, D.} and \textsc{Winkler, M.}:
	\newblock {\em Boundedness vs. Blow-up in a Chemotaxis System}.
	\newblock Journal of Differential Equations, 215(1):52--107, 2005.
	\newblock \href {http://dx.doi.org/10.1016/j.jde.2004.10.022}
	{\path{doi:10.1016/j.jde.2004.10.022}}.
	
	\bibitem{JagerExplosionsSolutionsSystem1992}
	\textsc{J{\"a}ger, W.} and \textsc{Luckhaus, S.}:
	\newblock {\em On Explosions of Solutions to a System of Partial Differential
		Equations Modelling Chemotaxis}.
	\newblock Transactions of the American Mathematical Society, 329(2):819--824,
	1992.
	\newblock \href {http://dx.doi.org/10.2307/2153966}
	{\path{doi:10.2307/2153966}}.
	
	\bibitem{MR1046835}
	\textsc{J\"{a}ger, W.} and \textsc{Luckhaus, S.}:
	\newblock {\em On explosions of solutions to a system of partial differential
		equations modelling chemotaxis}.
	\newblock Trans. Amer. Math. Soc., 329(2):819--824, 1992.
	\newblock \href {http://dx.doi.org/10.2307/2153966}
	{\path{doi:10.2307/2153966}}.
	
	\bibitem{keller1970initiation}
	\textsc{Keller, E.~F.} and \textsc{Segel, L.~A.}:
	\newblock {\em Initiation of slime mold aggregation viewed as an instability}.
	\newblock J. Theoret. Biol., 26(3):399--415, 1970.
	\newblock \href {http://dx.doi.org/10.1016/0022-5193(70)90092-5}
	{\path{doi:10.1016/0022-5193(70)90092-5}}.
	
	\bibitem{LSU}
	\textsc{Lady\v{z}enskaja, O.~A.}, \textsc{Solonnikov, V.~A.}, and
	\textsc{Ural'tseva, N.~N.}:
	\newblock {\em Linear and quasilinear equations of parabolic type}.
	\newblock Translated from the Russian by S. Smith. Translations of Mathematical
	Monographs, Vol. 23. American Mathematical Society, Providence, R.I., 1968.
	
	\bibitem{LUElliptic}
	\textsc{Ladyzhenskaya, O.~A.} and \textsc{Ural'tseva, N.~N.}:
	\newblock {\em Linear and quasilinear elliptic equations}.
	\newblock Translated from the Russian by Scripta Technica, Inc. Translation
	editor: Leon Ehrenpreis. Academic Press, New York-London, 1968.
	
	\bibitem{LankeitIrregularInitialData}
	\textsc{Lankeit, J.}
	\newblock {\em Immediate smoothing and global solutions for initial data in
		$L^1\times W^{1,2}$ in a Keller--Segel system with logistic terms in 2D},
	2020.
	\newblock [Preprint].
	\newblock URL: \url{https://arxiv.org/abs/2003.02644}.
	
	\bibitem{LankheitGNI}
	\textsc{Li, Y.} and \textsc{Lankeit, J.}:
	\newblock {\em Boundedness in a chemotaxis-haptotaxis model with nonlinear
		diffusion}.
	\newblock Nonlinearity, 29(5):1564--1595, 2016.
	\newblock \href {http://dx.doi.org/10.1088/0951-7715/29/5/1564}
	{\path{doi:10.1088/0951-7715/29/5/1564}}.
	
	\bibitem{LucaChemotacticSignalingMicroglia2003}
	\textsc{Luca, M.}:
	\newblock {\em Chemotactic {{Signaling}}, {{Microglia}}, and {{Alzheimer}}'s
		{{Disease Senile Plaques}}: {{Is There}} a {{Connection}}?}
	\newblock Bulletin of Mathematical Biology, 65(4):693--730, 2003.
	\newblock \href {http://dx.doi.org/10.1016/S0092-8240(03)00030-2}
	{\path{doi:10.1016/S0092-8240(03)00030-2}}.
	
	\bibitem{MarkPatterningNeuronalConnections1997}
	\textsc{Mark, M.~D.}, \textsc{Lohrum, M.}, and \textsc{P{\"u}schel, A.~W.}:
	\newblock {\em Patterning Neuronal Connections by Chemorepulsion: The
		Semaphorins}.
	\newblock Cell and Tissue Research, 290(2):299--306, 1997.
	\newblock \href {http://dx.doi.org/10.1007/s004410050934}
	{\path{doi:10.1007/s004410050934}}.
	
	\bibitem{SemiConductorExist}
	\textsc{Mock, M.~S.}:
	\newblock {\em An initial value problem from semiconductor device theory}.
	\newblock SIAM J. Math. Anal., 5:597--612, 1974.
	\newblock \href {http://dx.doi.org/10.1137/0505061}
	{\path{doi:10.1137/0505061}}.
	
	\bibitem{SemiConductorAsymptotic}
	\textsc{Mock, M.~S.}:
	\newblock {\em Asymptotic behavior of solutions of transport equations for
		semiconductor devices}.
	\newblock J. Math. Anal. Appl., 49:215--225, 1975.
	\newblock \href {http://dx.doi.org/10.1016/0022-247X(75)90172-9}
	{\path{doi:10.1016/0022-247X(75)90172-9}}.
	
	\bibitem{NagaiBlowupRadiallySymmetric1995}
	\textsc{Nagai, T.}:
	\newblock {\em Blow-up of Radially Symmetric Solutions to a Chemotaxis System}.
	\newblock Advances in Mathematical Sciences and Applications, 5(2):581--601,
	1995.
	
	\bibitem{MR1361006}
	\textsc{Nagai, T.}:
	\newblock {\em Blow-up of radially symmetric solutions to a chemotaxis system}.
	\newblock Adv. Math. Sci. Appl., 5(2):581--601, 1995.
	
	\bibitem{MR1887324}
	\textsc{Nagai, T.}:
	\newblock {\em Blowup of nonradial solutions to parabolic--elliptic systems
		modeling chemotaxis in two-dimensional domains}.
	\newblock J. Inequal. Appl., 6(1):37--55, 2001.
	\newblock \href {http://dx.doi.org/10.1155/S1025583401000042}
	{\path{doi:10.1155/S1025583401000042}}.
	
	\bibitem{NagaiBlowupNonradialSolutions2001}
	\textsc{Nagai, T.}:
	\newblock {\em Blowup of Nonradial Solutions to Parabolic\textendash{}Elliptic
		Systems Modeling Chemotaxis in Two-Dimensional Domains}.
	\newblock Journal of Inequalities and Applications, 2001(1):970292, 2001.
	\newblock \href {http://dx.doi.org/10.1155/S1025583401000042}
	{\path{doi:10.1155/S1025583401000042}}.
	
	\bibitem{MR1970697}
	\textsc{Nagai, T.}:
	\newblock {\em Global existence and blowup of solutions to a chemotaxis
		system}.
	\newblock In {\em Proceedings of the {T}hird {W}orld {C}ongress of {N}onlinear
		{A}nalysts, {P}art 2 ({C}atania, 2000)}, volume~47, pages 777--787, 2001.
	\newblock \href {http://dx.doi.org/10.1016/S0362-546X(01)00222-X}
	{\path{doi:10.1016/S0362-546X(01)00222-X}}.
	
	\bibitem{MR1623326}
	\textsc{Nagai, T.} and \textsc{Senba, T.}:
	\newblock {\em Global existence and blow-up of radial solutions to a
		parabolic--elliptic system of chemotaxis}.
	\newblock Adv. Math. Sci. Appl., 8(1):145--156, 1998.
	
	\bibitem{MR1610709}
	\textsc{Nagai, T.}, \textsc{Senba, T.}, and \textsc{Yoshida, K.}:
	\newblock {\em Application of the {T}rudinger-{M}oser inequality to a parabolic
		system of chemotaxis}.
	\newblock Funkcial. Ekvac., 40(3):411--433, 1997.
	
	\bibitem{MR1893940}
	\textsc{Osaki, K.} and \textsc{Yagi, A.}:
	\newblock {\em Finite Dimensional Attractor for One-Dimensional
		{{Keller}}-{{Segel}} Equations}.
	\newblock Funkcialaj Ekvacioj. Serio Internacia, 44(3):441--469, 2001.
	
	\bibitem{PiniChemorepulsionAxonsDeveloping1993}
	\textsc{Pini, A.}:
	\newblock {\em Chemorepulsion of Axons in the Developing Mammalian Central
		Nervous System}.
	\newblock Science, 261(5117):95--98, 1993.
	\newblock \href {http://dx.doi.org/10.1126/science.8316861}
	{\path{doi:10.1126/science.8316861}}.
	
	\bibitem{PorzioVespriHoelder}
	\textsc{Porzio, M.~M.} and \textsc{Vespri, V.}:
	\newblock {\em H\"{o}lder estimates for local solutions of some doubly
		nonlinear degenerate parabolic equations}.
	\newblock J. Differential Equations, 103(1):146--178, 1993.
	\newblock \href {http://dx.doi.org/10.1006/jdeq.1993.1045}
	{\path{doi:10.1006/jdeq.1993.1045}}.
	
	\bibitem{MR2483520}
	\textsc{Raczy\'{n}ski, A.}:
	\newblock {\em Stability property of the two-dimensional {K}eller--{S}egel
		model}.
	\newblock Asymptot. Anal., 61(1):35--59, 2009.
	
	\bibitem{RudinFunctionalAnalysis}
	\textsc{Rudin, W.}:
	\newblock {\em Functional analysis}.
	\newblock International Series in Pure and Applied Mathematics. McGraw-Hill,
	Inc., New York, second edition, 1991.
	
	\bibitem{SenbaParabolicSystemChemotaxis2001}
	\textsc{Senba, T.} and \textsc{Suzuki, T.}:
	\newblock {\em Parabolic System of Chemotaxis: {{Blowup}} in a Finite and the
		Infinite Time}.
	\newblock Methods and Applications of Analysis, 8(2):349--368, 2001.
	\newblock \href {http://dx.doi.org/10.4310/MAA.2001.v8.n2.a9}
	{\path{doi:10.4310/MAA.2001.v8.n2.a9}}.
	
	\bibitem{MR1909263}
	\textsc{Senba, T.} and \textsc{Suzuki, T.}:
	\newblock {\em Weak solutions to a parabolic--elliptic system of chemotaxis}.
	\newblock J. Funct. Anal., 191(1):17--51, 2002.
	\newblock \href {http://dx.doi.org/10.1006/jfan.2001.3802}
	{\path{doi:10.1006/jfan.2001.3802}}.
	
	\bibitem{ShortCrime}
	\textsc{Short, M.~B.}, \textsc{D'Orsogna, M.~R.}, \textsc{Pasour, V.~B.},
	\textsc{Tita, G.~E.}, \textsc{Brantingham, P.~J.}, \textsc{Bertozzi, A.~L.},
	and \textsc{Chayes, L.~B.}:
	\newblock {\em A statistical model of criminal behavior}.
	\newblock Math. Models Methods Appl. Sci., 18(suppl.):1249--1267, 2008.
	\newblock \href {http://dx.doi.org/10.1142/S0218202508003029}
	{\path{doi:10.1142/S0218202508003029}}.
	
	\bibitem{MR3936129}
	\textsc{Souplet, P.} and \textsc{Winkler, M.}:
	\newblock {\em Blow-up profiles for the parabolic--elliptic {K}eller--{S}egel
		system in dimensions {$n\geq 3$}}.
	\newblock Comm. Math. Phys., 367(2):665--681, 2019.
	\newblock \href {http://dx.doi.org/10.1007/s00220-018-3238-1}
	{\path{doi:10.1007/s00220-018-3238-1}}.
	
	\bibitem{CompetingChemotaxis}
	\textsc{Tao, Y.} and \textsc{Wang, Z.-A.}:
	\newblock {\em Competing effects of attraction vs. repulsion in chemotaxis}.
	\newblock Math. Models Methods Appl. Sci., 23(1):1--36, 2013.
	\newblock \href {http://dx.doi.org/10.1142/S0218202512500443}
	{\path{doi:10.1142/S0218202512500443}}.
	
	\bibitem{WinklerSemigroupRegularity}
	\textsc{Winkler, M.}:
	\newblock {\em Aggregation vs. global diffusive behavior in the
		higher-dimensional {K}eller--{S}egel model}.
	\newblock J. Differential Equations, 248(12):2889--2905, 2010.
	\newblock \href {http://dx.doi.org/10.1016/j.jde.2010.02.008}
	{\path{doi:10.1016/j.jde.2010.02.008}}.
	
	\bibitem{WinklerBlowUp}
	\textsc{Winkler, M.}:
	\newblock {\em Finite-time blow-up in the higher-dimensional
		parabolic-parabolic {K}eller--{S}egel system}.
	\newblock J. Math. Pures Appl. (9), 100(5):748--767, 2013.
	\newblock \href {http://dx.doi.org/10.1016/j.matpur.2013.01.020}
	{\path{doi:10.1016/j.matpur.2013.01.020}}.
	
	\bibitem{MR4022112}
	\textsc{Winkler, M.}:
	\newblock {\em How strong singularities can be regularized by logistic
		degradation in the {K}eller--{S}egel system?}
	\newblock Ann. Mat. Pura Appl. (4), 198(5):1615--1637, 2019.
	\newblock \href {http://dx.doi.org/10.1007/s10231-019-00834-z}
	{\path{doi:10.1007/s10231-019-00834-z}}.
	
	\bibitem{MR3905266}
	\textsc{Winkler, M.}:
	\newblock {\em Instantaneous regularization of distributions from
		{$(C^0)^\star\times L^2$} in the one-dimensional parabolic {K}eller--{S}egel
		system}.
	\newblock Nonlinear Anal., 183:102--116, 2019.
	\newblock \href {http://dx.doi.org/10.1016/j.na.2019.01.017}
	{\path{doi:10.1016/j.na.2019.01.017}}.
	
\end{thebibliography}
\end{document}